\documentclass[10pt,draftclsnofoot,onecolumn]{IEEEtran}

\usepackage{setspace}
\usepackage{mathptmx} 
\usepackage{amsmath,amssymb} 
\usepackage{dsfont}
\usepackage{verbatim}
\usepackage{multicol,longtable}
\usepackage{multirow}
\usepackage{graphics}
\usepackage{epsfig}
\usepackage{txfonts}
\usepackage{calc}
\usepackage{subfig}
\usepackage[ruled,vlined]{algorithm2e}
\usepackage{enumerate}
\usepackage{url}
\usepackage{algorithm2e}
\usepackage{array}

\newtheorem{theorem}{Theorem}[section]
\newtheorem{lemma}{Lemma}[section]
\newtheorem{corollary}{Corollary}[section]

\newtheorem{remark}{Remark}[section]
\newtheorem{proposition}{Proposition}[section]
\newtheorem{assumption}{Assumption}[section]

\newcolumntype{L}[1]{>{\raggedright\let\newline\\\arraybackslash\hspace{0pt}}m{#1}}
\newcolumntype{C}[1]{>{\centering\let\newline\\\arraybackslash\hspace{0pt}}m{#1}}
\newcolumntype{R}[1]{>{\raggedleft\let\newline\\\arraybackslash\hspace{0pt}}m{#1}}

\begin{document}

\title{Nonlinear Programming Methods for Distributed Optimization}


\author{Ion~Matei, \and John~S.~Baras
\thanks{Ion Matei  is with the Palo Alto Research Center (PARC), Palo Alto, CA 94304 ({\tt ion.matei@parc.com}) and John S. Baras is with the Institute for Systems Research at University of Maryland, College Park, MD 20742 ({\tt baras@umd.edu}).}
}

\maketitle

\begin{abstract}
 In this paper we investigate how standard nonlinear programming algorithms can be used to solve constrained optimization problems in a distributed manner. The optimization setup consists of a set of agents interacting through a communication graph that have as common goal the minimization of a function expressed as a sum of (possibly non-convex) differentiable functions. Each function in the sum corresponds to an agent and each agent has associated an equality constraint. By re-casting the distributed optimization problem into an equivalent, augmented centralized problem, we show that distributed algorithms result naturally from applying standard nonlinear programming techniques. Due to the distributed formulation, the standard assumptions and convergence results no longer hold. We emphasize what changes are necessary for convergence to still be achieved for three algorithms: two algorithms based on Lagrangian methods, and an algorithm based the method of multipliers. The changes in the convergence results are necessary mainly due to the fact that the (local) minimizers of the lifted optimization problem are not regular, as a results of the distributed formulation. Unlike the standard algorithm based on the method of multipliers, for the distributed version we cannot show that the theoretical superlinear convergence rate can be achieved.
\end{abstract}

\section{Introduction}
\label{sec_12011439}

Multi-agent, distributed optimization algorithms received a lot of attention in the recent years due to their applications in network
resource allocation, collaborative control, estimation and identification problems.  In these type of problems a group of agents has as common goal the optimization of a cost function under limited information and resources. The limited information may be induced by the fact that an agent can communicate with only a subset of the total set of agents, or/and by the fact that an agent is aware of only a part of the cost functions or constraint sets.

A distributed optimization algorithm was introduced in \cite{Nedic_2007}, where the convex optimization cost is expressed as a sum of functions and each function in the sum corresponds to an agent. In this formulation the agents interact with each other subject to a communication network, usually modeled as a undirected graph. The algorithm combines a standard (sub)gradient descent step with a consensus step; the latter being added to deal with the limited information about the cost function and about the actions of the agents. Extensions of this initial version followed in the literature. \cite{Nedic2011,Ram2010} include communication noise and errors on subgradients, \cite{5624570,5720506} assume a random communication graph, \cite{Nedic2011,Srivastava2011} study asynchronous versions of the algorithm,  \cite{Lobel2011} considers state-dependent communication topologies, while \cite{Cherukuri:2013} assumes directed communication graphs. Another modification of the algorithm described in \cite{Nedic_2007} was introduced in \cite{Johansson_2008}, where the authors change the order in which the consensus-step and the subgradient descent step are executed. The algorithms discussed above became popular in the signal processing community as well, being used for solving distributed filtering and parameter identification problems \cite{6197748,5373900}.
Consensus-based distributed optimization algorithms were further used to solve constrained convex optimization problems where all agents have the same constraint set \cite{Johansson2009SIAM,Nedic2011,Ram2010} or where each agent has its own set of constraints \cite{5404774, Srivastava2011}. Other approaches for obtaining distributed algorithms use dual decomposition \cite{Terelius:2011}, augmented Lagragian \cite{Jakovetic:2013, 5755209}, or in particular, distributed versions of the Alternating Direction Method of Multipliers (ADMM) algorithm \cite{Boyd:2011,Shi:2014,Wei:2012}. A summary of the relevant problem setups and approaches concerning distributed optimization is shown in Table \ref{tab_01051157}.
{\begin{table}[ht]
\footnotesize
\centering
{
\begin{tabular}{|C{1.5cm}|C{2.4cm}|C{2cm}|C{1.8cm}|C{3.5cm}|C{1cm}|}
\hline
\textbf{Cost type}          & \textbf{Constraints type}          & \multicolumn{2}{c|}{\textbf{Communication graph}}                                                                    & \textbf{Approach}                                                                          & \textbf{Authors} \\ \hline
\multirow{11}{*}{convex}    & \multirow{7}{*}{unconstrained}     & \multirow{5}{*}{time invariant} & \multirow{4}{*}{undirected}                                                        & combination of sub-gradient and consensus steps                                            &    \cite{Johansson_2008,Nedic_2007}              \\ \cline{5-6}
                            &                                    &                                 &                                                                                    & augmented Lagrangian method with a consensus step included in the inner step               &  \cite{Jakovetic:2013}                 \\ \cline{5-6}
                            &                                    &                                 &                                                                                    & randomized incremental subgradient method                                                  &   \cite{Johansson2009SIAM}               \\ \cline{5-6}
                            &                                    &                                 &                                                                                    & alternating direction method of multipliers                                                &   \cite{Boyd:2011,Shi:2014,Wei:2012}               \\ \cline{4-6}
                            &                                    &                                 & directed                                                                           & Laplacian-gradient dynamics                                                                &    \cite{Cherukuri:2013}              \\ \cline{3-6}
                            &                                    & \multirow{2}{*}{time varying}   & random                                                                             & combination of sub-gradient and consensus steps                                            &    \cite{5624570,5720506}              \\ \cline{4-6}
                            &                                    &                                 & state dependent                                                                    & combination of sub-gradient and consensus steps                                            &    \cite{Lobel2011}              \\ \cline{2-6}
                            & \multirow{2}{*}{global convex set} & \multirow{2}{*}{time invariant} & undirected                                                                         & combination of (projected) sub-gradient (with stochastic errors) and consensus steps       & \cite{Nedic2011,Ram2010}                 \\ \cline{4-6}
                            &                                    &                                 & undirected bounded communication delays & dual decomposition                                                                         &     \cite{Terelius:2011}             \\ \cline{2-6}
                            & \multirow{2}{*}{local convex sets } & \multicolumn{2}{c|}{\multirow{2}{*}{time invariant, undirected}}                                                     & combination of (projected) sub-gradient and consensus steps                                &      \cite{Nedic_2009}            \\ \cline{5-6}
                            &                                    & \multicolumn{2}{c|}{}                                                                                                & augmented Lagrangian method with a gossip step included in the inner step                  &     \cite{Jakovetic:2013}             \\ \hline
\multirow{2}{*}{non-convex} & global inequality constraints      & \multicolumn{2}{c|}{time varying, periodic strong connectivity}                                                      & approximate solution obtained using a Lagrangian duality method combined with consensus step &    \cite{Zhu:2013}              \\ \cline{2-6}
                            & local equality constraints         & \multicolumn{2}{c|}{time invariant, undirected}                                                                      & (augmented)Lagrangian methods                                                              &    \cite{Matei_CDC2013,Matei_CDC2014}              \\ \hline
\end{tabular}}
\caption{Problem setups and approaches for distributed optimization}
\label{tab_01051157}
\end{table}}

In this paper we study as well a distributed optimization problem whose goal is to minimize an objective function expressed as a sum of functions. Each function in the sum is associated to an agent that has assigned an equality constraint, as well. We propose three distributed algorithms: two first-order algorithms for solving the first order necessary optimality conditions and an algorithm inspired by the method of multipliers. The second first-order algorithm uses an augmented Lagrangian idea to obtain weaker conditions for local convergence.  The main message of this paper is that standard optimization techniques can be used to solve optimization problems in a distributed manner \textbf{as is}, provided appropriate changes in the convergence proofs are made to deal with the fact that the standard assumptions no longer hold as a result of lack of complete information. We  make no convexity assumptions on the cost and constraint functions, but we assume they are continuously differentiable. Consequently, our convergence results are local. Distributed algorithms for solving constrained, non-convex optimization problems were also proposed in \cite{Matei_CDC2013} and \cite{Zhu:2013}. This paper is based on two conference papers \cite{Matei_CDC2013,Matei_CDC2014} where only preliminary results are shown and most of the proofs were omitted due to space restrictions.

The paper is organized as follows: in Section \ref{sec_12011440} we formulate the constrained optimization problem while in Section \ref{sec_12011441} we introduced three distributed optimization algorithms for solving optimization problems with equality constraints. Section \ref{sec_12011442} presents the origins of the algorithms by demonstrating that our initial optimization problem is equivalent to a lifted optimization problem with equality constraints. Sections \ref{sec_12011443} and \ref{sec_12011444}. Some proofs and supporting results are included in the Appendix.

\emph{Notations and definitions}: For a matrix $A$, its $(i,j)$ entry is denoted by $[A]_{ij}$ and its transpose is given by $A'$. If $A$ is a symmetric matrix, $A\succ 0$ ($A \succeq 0$) means that $A$ is positive (semi-positive) definite. The nullspace and range of $A$ are denoted by $\textmd{Null}(A)$ and $\textmd{Range}(A)$, respectively. The symbol $\otimes$ is used to represent the Kronecker product between two matrices. The vector of all ones is denoted by $\mathds{1}$. Let $x$ and $Q$ be a vector and a set of vectors, respectively. By $x+Q$ we understand the set of vectors produced by adding $x$ to each element of $Q$, that is, $x+Q\triangleq \{x+y\ |\ y\in Q\}$. Let $\| \cdot\|$ be a vector norm. By $\|x-Q\|$ we denote the distance between the vector $x$ and the set $Q$, that is, $\|x-Q\|\triangleq \inf_{y\in Q}\|x-y\|$. Let $f:\mathds{R}^n\rightarrow \mathds{R}$ be a function. We denote by $\nabla f(x)$ and by $\nabla^2 f(x)$ the gradient and the Hessian of $f$ at $x$, respectively. Let $F:\mathds{R}^n\times \mathds{R}^m\rightarrow \mathds{R}$ be a function of variables $(x,y)$. The block descriptions of the gradient and of the Hessian of $F$ at $(x,y)$ are given by $\nabla F(x,y)'=\left(\nabla_x F(x,y)',\nabla_y F(x,y)'\right)$, and
$$\footnotesize \nabla^2 F(x,y)=\left( \begin{array}{cc}
\nabla_{xx}^2F(x,y) & \nabla_{xy}^2 F(x,y)\\
\nabla_{xy}^2 F(x,y) & \nabla_{yy}^2F(x,y)
\end{array}
\right),$$
respectively.
Let $\{A_i\}_{i=1}^N$ be a set of matrices. By $\textmd{diag}(A_i,\ i=1,\ldots,N)$ we understand a block diagonal matrix, where the $i^{\textmd{th}}$ block matrix is given by $A_i$. We say that the set $\mathcal{X}$ is an \emph{attractor} for the dynamics $x_{k+1}=f(x_k)$, if there exists $\epsilon>0$, such that for any $x_0\in S_{\epsilon}$, with $S_{\epsilon} = \left\{x\ |\ \|x-\mathcal{X}\|<\epsilon\right\}$,  $\lim_{k\rightarrow \infty}\|x_k-\mathcal{X}\|=0$.

\section{Problem description}
\label{sec_12011440}

In this section we describe the setup of our problem.  We present first the communication model followed by the optimization model. 
\subsection{Communication model}
 A set of $N$ agents interact with each other through a communication topology modeled as an undirected communication graph $\mathcal{G}=(\mathcal{V},\mathcal{E})$, where $\mathcal{V}=\{1,2,\ldots,N\}$ is the set of nodes and $\mathcal{E}=\{e_{ij}\}$ is the set of edges. An edge between two nodes $i$ and $j$ means that agents $i$ and  $j$ can exchange information (or can cooperate). We assume that at each time instant $k$ the agents can synchronously exchange information with their neighbors. We denote by $\mathcal{N}_i\triangleq \{j\ |\ e_{ij}\in \mathcal{E}\}$ the set of neighbors of agent $i$. Consider the set of pairs $\{(i,j), j\in \mathcal{N}_i,\ i=1,\ldots,N\}$ and let $\bar{N} = \sum_{i=1}^N|\mathcal{N}_i|$, where $|\cdot|$ denotes the cardinality of a set. We denote by $S\in \mathds{R}^{\bar{N}\times N}$ the (weighted) edge-node incidence matrix of graph $\mathcal{G}$, for which each row number corresponds to a unique pair $(i,j)$ from the previously defined set. The matrix $S$ is defined as
 {\small \begin{equation}
 \label{equ_02112218}
 [S]_{(ij),l}=\left\{\begin{array}{ll}
 s_{ij} & i = l,\\
-s_{ij} & j=l,\\
 0 & \textmd{otherwise},
 \end{array}
 \right.
 \end{equation}}
 where $s_{ij}$ are given positive scalars.

 \begin{remark}
 \label{rem_03081940}
 It is not difficult to observe that the matrix $L = S'S = (l_{ij})$ is a (weighted) Laplacian matrix corresponding to the graph $\mathcal{G}$ and $\textmd{Null}(L) = \textmd{Null}(S)$. Moreover, for any $i\neq j$ we have that $l_{ij} = -(s_{ij}^2+s_{ji}^2)$. $\Box$
 \end{remark}
 In the next sections we are going to make use of a set of properties of the matrices $S$ and $L$; properties that are grouped in what follows.
\begin{proposition}
\label{pro_10151123}
The matrix $S$ and $L$ defined with respect to a connected graph $\mathcal{G}$ satisfies the following properties:
\begin{enumerate}[(a)]
\item The nullspaces of $S$ and $L$ are given by $\textmd{Null}(S)=\textmd{Null}(L)=\{\gamma \mathds{1}\ |\ \gamma \in \mathds{R}\}$;
\item Let $\mathbf{S} = S\otimes I$ and $\mathbf{L} = L\otimes I$, where $I$ is the $n$-dimensional identity matrix. Then the nullspaces of $\mathbf{S}$ and $\mathbf{L}$ are given by $\textmd{Null}(\mathbf{S})=\textmd{Null}(\mathbf{L})=\{ \mathds{1}\otimes x \ |\ x \in \mathds{R}^n\}$.
\item Let $\boldsymbol{\lambda}$ be a vector in $\mathds{R}^{n\bar{N}}$ and let $\{{u}_1,...,{u}_m\}$ be an orthogonal basis spanning $\textmd{Null}({S}')$. Then the orthogonal projection of $\boldsymbol{\lambda}$ on $\textmd{Null}\left(\mathbf{S}'\right)$ is given by $\boldsymbol{\lambda}_{\perp} = \mathbf{J}\boldsymbol{\lambda}$, where $\mathbf{J}$ is the orthogonal projection matrix (operator) defined as
\begin{equation*}
\label{equ_10411123}
\mathbf{J}\triangleq J\otimes I,
\end{equation*}
where $J = UU'$ with $U = [u_1,...,u_m]$ is the orthogonal projection operator on $\textmd{Null}(S')$.
\end{enumerate}
\end{proposition}

\subsection{Optimization model}

We consider a function $f:\mathds{R}^n \rightarrow \mathds{R}$ expressed as a sum of $N$ functions $f(x) = \sum_{i=1}^Nf_i(x)$, and a vector-valued function $h:\mathds{R}^n \rightarrow \mathds{R}^N$ where $h\triangleq (h_1,h_2,\ldots,h_N)'$, with $h_i:\mathds{R}^n\rightarrow \mathds{R}$ and $N\leq n$.

We make the following assumptions on the functions $f$ and $h$ and on the communication model.
\begin{assumption}
\label{assu_02091524}
\begin{enumerate}[(a)]
\item The functions $f_i(x)$ and $h_i(x)$, $i=1,\ldots,N$ are twice continuously differentiable;
\item Agent $i$ has knowledge of only functions $f_i(x)$ and $h_i(x)$, and scalars $s_{ij}$, for $j\in \mathcal{N}_i$;
 \item Agent $i$ can exchange information only with agents in the set of neighbors defined by $\mathcal{N}_i$;
\item The communication graph $\mathcal{G}$ is connected.
\end{enumerate}
\end{assumption}

The common goal of the agents is to minimize the following optimization problem with equality constraints
{\small \begin{eqnarray}
\nonumber
(\texttt{P}_1) & \min_{x\in \mathds{R}^n} & f(x),\\
\nonumber
& \textmd{subject to:}& h(x) = 0,
\end{eqnarray}}
under Assumptions \ref{assu_02091524}. Throughout the rest of the paper we assume that problem $(\texttt{P}_1)$ has at least one local minimizer.

Let $x^*$ be a local minimizer of $(\texttt{P}_1)$ and let $\nabla h\left(x^*\right)\triangleq \left[\nabla h_1\left(x^*\right),\nabla h_2\left(x^*\right),\ldots,\nabla h_N\left(x^*\right)\right]$
be a matrix whose columns are the gradients of the functions $h_i(x)$ computed at $x^*$. The following assumption is used to guarantee the uniqueness of the Lagrange multiplier vector $\psi^*$ appearing in the first order necessary conditions of $(\texttt{P}_1)$, namely $\nabla f(x^*) + \nabla h\left(x^*\right)\psi^* =0$.
\begin{assumption}
\label{assu_02091703}
Let $x^*$ be a local minimizer of $(\texttt{P}_1)$. The matrix $\nabla h\left(x^*\right)$ is full rank, or equivalently, the vectors $\left\{\nabla h_i\left(x^*\right)\right\}_{i=1}^N$ are linearly independent.
\end{assumption}
Together with some additional assumptions on $f(x)$ and $h(x)$, Assumption \ref{assu_02091703} is typically used to prove local convergence for the ``original'' method of multipliers applied to Problem $(\texttt{P}_1)$ (see for example Section 2.2, page 104 of \cite{Bertsekas1982}). As we will see in the next sections, the same assumption will be used to prove local convergence for a distributed version of the method of multipliers used to solved a `lifted` optimization problem with equality constrains.

\begin{remark}
\label{rem_02091543}
We assumed that each agent has an equality constraint of the type $h_i(x) = 0$. All the results presented in what follows can be easily adapted for the case where only $m\leq N$ agents have equality constraints.
\end{remark}

\section{Main results}
\label{sec_12011441}
In this section we present the main results of the paper, namely three distributed algorithms for solving $(\texttt{P}_1)$. As seen later in the paper, these algorithm are a result of applying Lagrangian methods and the method of multipliers to an augmented but equivalent version of $(\texttt{P}_1)$. The algorithms are used to solve the equivalent problem in a centralized manner, however due to the nature of the cost function and the constraints, they can be naturally implemented in a distributed manner.

\subsection{Distributed Algorithms}
Let $x^*$ be a local minimizer of $(\texttt{P}_1)$  and let $x_{i,k}$ denote agent $i$'s \emph{estimate} of $x^*$, at time-slot $k$. In addition, let us denote by $\mathbf{x}_k\in \mathds{R}^{nN}$, $\boldsymbol{\mu}_k\in \mathds{R}^{N}$ and $\boldsymbol{\lambda}_k\in \mathds{R}^{n\bar{N}}$ the vectors $\mathbf{x}_{k} =\left(x_{i,k}\right)$, $\boldsymbol{\mu}_k=(\mu_{i,k})$ and $\boldsymbol{\lambda}_k = (\lambda_{i,k})$, with $\lambda_{i,k} = (\lambda_{ij,k})$ for all $j\in \mathcal{N}_i$. The first algorithm based on Lagrangian methods for solving $(\texttt{P}_1)$, denoted by Algorithm $(\texttt{A}_1)$ is given by the following iterations:
\begin{eqnarray}
\label{equ_02091554}
{x}_{i,k+1} &=& {x}_{i,k} - \alpha \nabla f_i(x_{i,k})-\alpha \mu_{i,k}\nabla h_i(x_{i,k})-\\
\nonumber
&-&\alpha \sum_{j\in \mathcal{N}_i}\left(s_{ij}\lambda_{ij,k}-s_{ji}\lambda_{ji,k}\right),\ x_{i,0} = x_{i}^0,\\
\label{equ_02091555}
\mu_{i,k+1} &=&\mu_{i,k}+\alpha h_i(x_{i,k}),\ \mu_{i,0}=\mu_i^0,\\
\label{equ_02091556}
\lambda_{ij,k+1} &=& \lambda_{ij,k}+\alpha \left(s_{ij}x_{i,k}-s_{ji}x_{j,k}\right),\ \lambda_{ij,0}=\lambda_{ij}^0,\ j\in \mathcal{N}_i,
\end{eqnarray}
where $\alpha>0$ is the step-size of the algorithm, $\nabla f_i(x_{i,k})$ and  $\nabla h_i(x_{i,k})$  denote the gradients of functions $f_i(x)$ and $h_i(x)$, respectively, computed at $x_{i,k}$, and $x_{i}^0$, $\mu_i^0$ and $\lambda_{ij}^0$ are given scalars. In addition, the positive scalars $s_{ij}$ are the entries of the incidence matrix $S$ of the graph $\mathcal{G}$ defined in (\ref{equ_02112218}).

In the second algorithm referred to as Algorithm $(\texttt{A}_2)$, the iteration for updating agent $i^{\textmd{th}}$ estimate $x_{i,k}$ has two additional terms that will allow for weaker assumptions for proving convergence of the algorithm, compared to the previous algorithm:
\begin{eqnarray}
\nonumber
  {x}_{i,k+1}& = &{x}_{i,k} - \alpha \nabla f_i(x_{i,k})-\alpha \mu_{i,k}\nabla h_i(x_{i,k})-\\
\nonumber
  &-&\alpha \sum_{j\in \mathcal{N}_i}\left(s_{ij}\lambda_{ij,k}-s_{ji}\lambda_{ji,k}\right)-\alpha c  h_i(x_{i,k})\nabla h_i(x_{i,k})\\
\label{equ_02181536}
&-&\alpha c \sum_{j\in \mathcal{N}_i}l_{ij}(x_{i,k}-x_{j,k}),\ x_{i,0} = x_{i}^0,\\
\label{equ_02181537}
\mu_{i,k+1} &=&\mu_{i,k}+\alpha h_i(x_{i,k}),\ \mu_{i,0}=\mu_i^0,\\
\label{equ_02181538}
\lambda_{ij,k+1} &=& \lambda_{ij,k}+\alpha \left(s_{ij}x_{i,k}-s_{ji}x_{j,k}\right),\ \lambda_{ij,0}=\lambda_{ij}^0,\ j\in \mathcal{N}_i,
\end{eqnarray}
where in addition to the parameters of Algorithm $(\texttt{A}_1)$, we have a new positive parameter $c$, and $l_{ij} = s_{ij}^2+s_{ji}^2$.

Finally the third distributed algorithm denoted by Algorithm $(\texttt{A}_3)$ has at its origin the method of multipliers applied to the augmented and equivalent version of $(\texttt{P}_1)$. The iterations of the algorithm are given by
{\small \begin{eqnarray}
\nonumber
\mathbf{x}_{k} &=&\arg\min_{\bf{x}} \sum_{i}f_i(x_i)+\mu_{i,k}h_i(x_i)+\\
\nonumber
&&+\sum_{j\in \mathcal{N}_i}\lambda_{ij,k}s_{ij}(x_i-x_j)+\frac{c_k}{2}h_i(x_i)^2+\\
\label{equ_02091554}
&& +\frac{c_k}{2}\sum_{j\in \mathcal{N}_i} (s_{ij}^2+s_{ji}^2)x_i(x_i-x_j),\ \mathbf{x}_{0} = \mathbf{x}^0,\\
\label{equ_02091555}
\mu_{i,k+1} &=&\mu_{i,k}+c_k h_i(x_{i,k}),\ \mu_{i,0}=\mu_i^0,\ i=1,\ldots,N,\\
\label{equ_02091556}
\lambda_{ij,k+1} &=& \lambda_{ij,k}+c_k s_{ij}(x_{i,k}-x_{j,k}),\ \lambda_{ij,0}=\lambda_{ij}^0,\ j\in \mathcal{N}_i,
\end{eqnarray}}
where $\{c_k\}$ is an non-decreasing sequence of positive numbers known by all agents. Note that at each time instant $k$ we need to solve the unconstrained optimization problem (\ref{equ_02091554}). For Algorithm $(\texttt{A}_3)$ to be distributed we need to provide a distributed algorithm that solves (\ref{equ_02091554}). Due to the structure of the cost function in (\ref{equ_02091554}), such an algorithm results from using a gradient-descent method, namely
{\small \begin{eqnarray}
\nonumber
x_{i,\tau+1}&=& x_{i,\tau} - \alpha_{\tau}\left[\nabla f_i(x_{i,\tau})+\mu_{i,k}\nabla h_i(x_{i,\tau}) +\right. \\
\nonumber
&+& \sum_{j\in \mathcal{N}_i}(s_{ij}\lambda_{ij,k}-s_{ji}\lambda_{ji,k})+c_k\nabla h_i(x_{i,\tau}) h_i(x_{i,\tau})\\
\label{equ_03121344}
&+&c_k \sum_{j\in \mathcal{N}_i}(s_{ij}^2+s_{ji}^2)(x_{i,\tau}-x_{j,\tau})], \ i=1,\ldots,N
\end{eqnarray}}
with $x_{i,0} = (\mathbf{x}_k)_i$, and where $\{\alpha_{\tau}\}$ is a globally known sequence of step-sizes for the iteration (\ref{equ_03121344}). Note that we denote by $\tau$ the iteration index for the algorithm used to solve (\ref{equ_02091554}).

\subsection{Considerations on the distributed algorithms}

Algorithms $(\texttt{A}_1)$, $(\texttt{A}_2)$ are part of the general class of methods, called Lagrangian methods (see for example Section 4.4.1, page 386, \cite{Ber99}). They are based on a first order method, and therefore they achieve a linear rate of convergence, while Algorithm $(\texttt{A}_3)$ is based on the method of multipliers that theoretically can reach super-linear rate of convergence. In the case of all three algorithms we assume that the stepsize $\alpha$ and the sequences of stepsizes $\{c_k\}_{k\geq 0}$ and $\{\alpha_{\tau}\}_{\tau\geq 0}$ are globally known by all agents. It can be observed that the algorithms are indeed distributed since for updating their local variables $x_{i,k}$, $\mu_{i,k}$ and $\lambda_{ij,k}$ they use only local information ($ \nabla f_i(x_{i,k})$ and  $\nabla h_i(x_{i,k})$) and information from their neighbors ($x_{j,k}$, $\lambda_{ji,k}$, and $s_{ji}$ for $j\in \mathcal{N}_i$). In the case of Algorithm $(\texttt{A}_1)$, equation (\ref{equ_02091554}) describing the minimizer estimate update is comprised of a standard gradient descent step and two additional terms used to cope with the local equality constraint and the lack of complete information. Intuitively, $\mu_{i,k}$ can be seen as the price paid by agent $i$ for satisfying the local equality constraint, while $\lambda_{i,k}$ is the price paid by the same agent for having its estimate $x_{i,k}$ far away from the estimates of its neighbors. In the case of Algorithms $(\texttt{A}_2)$, $(\texttt{A}_3)$, the iteration for computing the minimizer has two additional terms. These terms have their origin in the use of an augmented Lagrangian and ensure the local convergence to a local minimizer under weaker conditions. In the next sections we will focus on proving converge results corresponding to the three algorithms. Although these three algorithms result from applying standard methods, their convergence results are no longer standard since the local minimizers corresponding to $(\texttt{P}_1)$ are no longer regular due to the distributed setup. We would like to emphasize that unless the communication topology is assumed \emph{undirected}, the updates of the minimizer estimates can no longer be implemented in a distributed manner since the agents would require information from agents not in their neighborhoods.

\section{An equivalent optimization problem with equality constraints}
\label{sec_12011442}
In this section we define an augmented optimization problem, from whose solution we can extract the solution of problem $(\texttt{P}_1)$. As made clear in what follows, the distributed algorithms proposed in this paper, follow from applying standard techniques for solving optimization problem with equality constraints to this particular problem.

Let us define the function $\mathbf{F}:\mathds{R}^{nN}\rightarrow \mathds{R}$ given by $\mathbf{F}(\mathbf{x}) = \sum_{i=1}^Nf_i(x_i),$
where $\mathbf{x}'=(x_1',x_2',\ldots,x_N')$, with $x_i\in \mathds{R}^n$. In addition we introduce the vector-valued functions $\mathbf{h}:\mathds{R}^{nN}\rightarrow \mathds{R}^{N}$ and $\mathbf{g}:\mathds{R}^{nN}\rightarrow \mathds{R}^{nN}$, where $\mathbf{h}(\mathbf{x})=\left(\mathbf{h}_1(\mathbf{x}),
\mathbf{h}_2(\mathbf{x}),\ldots,\mathbf{h}_N(\mathbf{x})\right)',$
with $\mathbf{h}_i:\mathds{R}^{nN}\rightarrow \mathds{R}$ given by $\mathbf{h}_i(\mathbf{x})=h_i(x_i)$,
and $\mathbf{g}(\mathbf{x})' = \left(g_1(\mathbf{x})', g_2(\mathbf{x})', \ldots, g_N(\mathbf{x})'\right),$
with $g_i:\mathds{R}^{nN}\rightarrow \mathds{R}^{|\mathcal{N}_i| n}$ given by $g_i(\mathbf{x}) = (g_{ij}(\mathbf{x})),$
where $g_{ij}(\mathbf{x}) = s_{ij}(x_i-x_j),$ with $s_{ij}$ positive scalars. The vector valued function $\mathbf{g}(\mathbf{x})$ can be compactly expressed as $g(\mathbf{x}) = \mathbf{S} \mathbf{x},$ where $\mathbf{S} = S\otimes I$, with $I$ the $n$-dimensional identity matrix and $S$ defined in (\ref{equ_02112218}).
We introduce the optimization problem
{\small \begin{eqnarray}
\label{equ_15391112}
(\texttt{P}_2) & \min_{\mathbf{x}\in \mathds{R}^{nN}} & \mathbf{F}(\mathbf{x}),\\
\label{equ_15391113}
& \textmd{subject to: } & \mathbf{h}(\mathbf{x})=0,\\
\label{equ_15391114}
&  & \mathbf{g}(\mathbf{x})=\mathbf{S}\mathbf{x}=0.
\end{eqnarray}}
 The Lagrangian function of \texttt{Problem} $(\texttt{P}_2)$ is a function $\boldsymbol{\mathcal{L}}:\mathds{R}^{nN}\times \mathds{R}^{N}\times\mathds{R}^{n\bar{N}}\rightarrow \mathds{R}$, defined as
{\small \begin{equation}
\label{equ_02111445}
\boldsymbol{\mathcal{L}}
\left(\mathbf{x},\boldsymbol{\mu},\boldsymbol{\lambda}\right)\triangleq \mathbf{F}(\mathbf{x})+
\boldsymbol{\mu}'\mathbf{h}(\mathbf{x})+\boldsymbol{\lambda}'\mathbf{S}\mathbf{x}.
\end{equation}}
We define also the augmented Lagrangian of problem $(\texttt{P}_2)$:
\begin{equation}
\label{equ_02181334}
\boldsymbol{\mathcal{L}}_c(\mathbf{x},\boldsymbol{\mu},\boldsymbol{\lambda}) = \mathbf{F}(\mathbf{x})+\boldsymbol{\mu}'\mathbf{h}(\mathbf{x})+
\boldsymbol{\lambda}'\mathbf{S}\mathbf{x}+\frac{c}{2}\|\mathbf{h}(\mathbf{x})\|^2+ \frac{c}{2}\mathbf{x}'\mathbf{L}\mathbf{x},
\end{equation}
 where $\mathbf{L} = \mathbf{S}'\mathbf{S}$ is a Laplacian type of matrix and $c$ is a positive scalar.
The gradient and the Hessian of $\boldsymbol{\mathcal{L}}_c(\mathbf{x},\boldsymbol{\mu},\boldsymbol{\lambda})$ are given by
\begin{equation}
\label{equ_02181504}
\nabla_{\mathbf{x}}\boldsymbol{\mathcal{L}}_c(\mathbf{x},\boldsymbol{\mu},\boldsymbol{\lambda}) =\nabla \mathbf{F}(\mathbf{x})+
\nabla\mathbf{h}(\mathbf{x})\boldsymbol{\mu}+
\mathbf{S}'\boldsymbol{\lambda}+c\nabla\mathbf{h}(\mathbf{x})\mathbf{h}(\mathbf{x})+c \mathbf{L}\mathbf{x},
\end{equation}
and
\begin{equation}
\label{equ_02181511}
\nabla_{\mathbf{x}\mathbf{x}}^2\boldsymbol{\mathcal{L}}_c(\mathbf{x},\boldsymbol{\mu},\boldsymbol{\lambda}) = \nabla^2\mathbf{F}(\mathbf{x})+\sum_{i=1}^N\boldsymbol{\mu}_i\nabla^2\mathbf{h}_i(\mathbf{x})+c\mathbf{L}+$$
$$+c\sum_{i=1}^N\left(\mathbf{h}_i(\mathbf{x})\nabla^2\mathbf{h}_i(\mathbf{x})+ \nabla\mathbf{h}_i(\mathbf{x})\nabla\mathbf{h}_i(\mathbf{x})'\right),
\end{equation}
respectively.

The following proposition states that by solving $(\texttt{P}_2)$ we solve in fact $(\texttt{P}_1)$ as well, and vice-versa.
\begin{proposition}
\label{pro_14411114}
 Let Assumptions \ref{assu_02091524} hold. The vector $x^*$ is a local minimizer of $(\texttt{P}_1)$ if and only if  $\mathbf{x}^* = \mathds{1}\otimes x^*$ is a local minimizer of $(\texttt{P}_2)$.
\end{proposition}

\begin{IEEEproof}
Since the Laplacian $L$ corresponds to a connected graph, according to Proposition \ref{pro_10151123}-(c), the nullspace of $\mathbf{S}$ is given by $\textmd{Null}(\mathbf{S})=\{ \mathds{1}\otimes x \ |\ x \in \mathds{R}^n\}$. From the equality constraint (\ref{equ_15391114}), we get that any local minimizer $\mathbf{x}^*$ of $(\texttt{P}_2)$ must be of the form $\mathbf{x}^* = \mathds{1}\otimes x^*$, for some $x^*\in \mathds{R}^n$. Therefore, the solution of $(\texttt{P}_2)$ must be searched in the set of vectors with structure given by $\mathbf{x} = \mathds{1}\otimes x$. Applying
this constraint, the cost function (\ref{equ_15391112}) becomes
$$\mathbf{F}(\mathbf{x}) = \sum_{i=1}^Nf_i(x)=f(x),$$
and the equality constraint (\ref{equ_15391113}) becomes
$$\mathbf{h}(\mathbf{x}) = h(x)=0,$$
which shows that we have recovered the optimization problem $(\texttt{P}_1)$.
\end{IEEEproof}

\begin{remark}
We note from above the importance of having a connected communication topology. Indeed, if $\mathcal{G}$ is not connected, then the nullspace of $\mathbf{S}$ is much richer than the subspace $\{ \mathds{1}\otimes x \ |\ x \in \mathds{R}^n\}$, and therefore the solution of $(\texttt{P}_2)$ may not necessarily be of the form $\mathbf{x}^* = \mathds{1}\otimes x^*$. However, the fact that we search a solution of $(\texttt{P}_2)$ of this particular structure is \emph{fundamental} for showing the equivalence between the two optimization problems.
\end{remark}

\section{Convergence analysis of the distributed algorithms based on Lagrangian methods}
\label{sec_12011443}
In this section we study the convergence properties of the two distributed algorithms for solving problem $(P_1)$ that are based on Lagrangian methods. In particular, they are obtained by applying a first order method for solving the first-order necessary optimality conditions, where in the case of the second algorithm, the first order necessary conditions are derived in terms of the augmented Lagrangian. As we will see next, this will require weaker conditions for convergence.

\subsection{Supporting results for the convergence analysis}
This section introduces a set of results used for the convergence analysis of the algorithms. The proofs of these results together with their auxiliary results can be found in the Appendix section.

We first characterize the tangent cone at a local minimizer of $(\texttt{P}_2)$ in terms of the tangent cone at a local minimizer of $(\texttt{P}_1)$.
\begin{proposition}
\label{pro_02181835}
Let Assumptions \ref{assu_02091524}-(a) and \ref{assu_02091703} hold, let $\mathbf{x}^* = \mathds{1}\otimes x^*$ be a local minimizer of $(\texttt{P}_2)$ and let $\boldsymbol{\Omega}$ denote the constraint set, that is, $\boldsymbol{\Omega} = \{\mathbf{x}\ |\ \mathbf{h}(\mathbf{x})=0,\mathbf{S}\mathbf{x}=0\}$. Then the tangent cone to $\boldsymbol{\Omega}$ at $\mathbf{x}^*$ is given by
$\textmd{TC}(\mathbf{x}^*,\boldsymbol{\Omega}) =\textmd{Null}\left(\left[\nabla \mathbf{h}(\mathbf{x}^*),\mathbf{S}'\right]'\right)=\left\{\mathds{1}\otimes h\ |\ h\in \textmd{TC}({x}^*,{\Omega})=\textmd{Null}\left(\nabla h(x^*)'\right) \right\}.$
\end{proposition}

Under the assumption that the matrix $\nabla h(x^*)$ is full rank, the first order necessary conditions of $(P_1)$ are given by $\nabla f(x^*)+\nabla h(x^*) \psi^*=0$, $h(x^*)=0$, where the vector $\psi^*$ is unique (see for example Proposition 3.3.1, page 255, \cite{Ber99}). An interesting question is whether or not there is a connection between $\psi^*$ and $\boldsymbol{\mu}^*$ shown in the first order necessary conditions of $(P_2)$. As shown in the following, the two vectors are in fact equal.

\begin{proposition}
\label{equ_02111222}
Let Assumptions \ref{assu_02091524} and \ref{assu_02091703} hold, let $\mathbf{x}^* = \mathds{1}\otimes x^*$ be a local minimizer of $(P_2)$ and let $\psi^*$ and $\boldsymbol{\mu}^*$ be the unique Lagrange multiplier vectors corresponding to the first order necessary conditions of $(P_1)$ and $(P_2)$, respectively. Then $\psi^* = \boldsymbol{\mu}^*$.
\end{proposition}

Using augmented Lagrangian for improving convergence was first study in \cite{Hestenes1969}. The basic idea is that under some assumptions on the standard Lagrangian function, the augmented Lagrangian can be made positive definite positive definite and therefore invertible when choosing a scalar $c$ large enough. The next proposition states that this property holds in our setup as well.
\begin{proposition}
\label{pro_03051237}
Let $(\mathbf{x}^*,\boldsymbol{\mu}^*,\boldsymbol{\lambda}^*)$ be a local minimizer-Lagrange multipliers pair of $(\texttt{P}_2)$ and assume that $\mathbf{z}'\nabla_{\mathbf{x}\mathbf{x}}^2\boldsymbol{\mathcal{L}}_0(\mathbf{x}^*,\boldsymbol{\mu}^*,\boldsymbol{\lambda}^*)\mathbf{z}>0$ for all $\mathbf{z}\in \textmd{TC}(\mathbf{x}^*,\boldsymbol{\Omega})$. Then there exists a positive scalar $\bar{c}$, such that $\nabla_{\mathbf{x}\mathbf{x}}^2\boldsymbol{\mathcal{L}}_c(\mathbf{x}^*,\boldsymbol{\mu}^*,\boldsymbol{\lambda}^*)\succ 0$ for all $c\geq \bar{c}$.
\end{proposition}
\begin{IEEEproof}
 We recall that the the Hessian $\nabla_{\mathbf{x}\mathbf{x}}^2\boldsymbol{\mathcal{L}}_c(\mathbf{x}^*,\boldsymbol{\mu}^*,\boldsymbol{\lambda}^*)$ is given by
$$\nabla_{\mathbf{x}\mathbf{x}}^2\boldsymbol{\mathcal{L}}_c(\mathbf{x}^*,\boldsymbol{\mu}^*,\boldsymbol{\lambda}^*) = \nabla_{\mathbf{x}\mathbf{x}}^2\boldsymbol{\mathcal{L}}_0(\mathbf{x}^*,\boldsymbol{\mu}^*,\boldsymbol{\lambda}^*)+c \nabla \mathbf{h}(\mathbf{x}^*)\nabla \mathbf{h}(\mathbf{x}^*)'+c\mathbf{S}'\mathbf{S}.$$
We have that $\mathbf{z}'\left[\nabla \mathbf{h}(\mathbf{x}^*)\nabla \mathbf{h}(\mathbf{x}^*)'+\mathbf{S}'\mathbf{S}\right]\mathbf{z}=0$ if and only if $z\in \textmd{Null}\left(\left[\nabla \mathbf{h}(\mathbf{x}^*),\mathbf{S}'\right]'\right)$. By Proposition \ref{pro_02181835} we have $\textmd{Null}\left(\left[\nabla \mathbf{h}(\mathbf{x}^*),\mathbf{S}'\right]'\right) =\textmd{TC}\left(\mathbf{x}^*,\boldsymbol{\Omega}\right)$. Finally, using Lemma 1.25, page 68 of \cite{Bertsekas1982}, there exists $\bar{c}>0$ so that $\nabla_{\mathbf{x}\mathbf{x}}^2\boldsymbol{\mathcal{L}}_c(\mathbf{x}^*,\boldsymbol{\mu}^*,\boldsymbol{\lambda}^*)\succ 0$, for all $c\geq \bar{c}$.
\end{IEEEproof}

\subsection{Lagrangian methods  - Algorithm $(A_1)$}

To find a solution of problem $(P_2)$ the first thing we can think of is solving the set of necessary conditions:
\begin{eqnarray}
\label{equ_18301114}
\nabla\mathbf{F}(\mathbf{x})+\mathbf{S}' \boldsymbol{\lambda}+ \nabla \mathbf{h}(\mathbf{x})\boldsymbol{\mu} &=& 0,\\
\label{equ_18311113}
\mathbf{h}(\mathbf{x})&=&0,\\
\label{equ_18311114}
\mathbf{S}\mathbf{x} &=& 0.
\end{eqnarray}
Solving (\ref{equ_18301114})-(\ref{equ_18311114}) does not guarantee finding a local minimizer, but at least the local minimizers are among the solutions of the above nonlinear system of equations.
An approach for solving (\ref{equ_18301114})-(\ref{equ_18311114}) consists of using a first order method (see for instance Section 4.4.1, page 386, \cite{Ber99}), which is given by
\begin{eqnarray}
\label{equ_18331114}
\mathbf{x}_{k+1} &=& \mathbf{x}_k - \alpha \left[\nabla\mathbf{F}(\mathbf{x}_k)+\nabla \mathbf{h}(\mathbf{x}_k)\boldsymbol{\mu}_k+\mathbf{S}' \boldsymbol{\lambda}_k\right],\\
\label{equ_18341113}
\boldsymbol{\mu}_{k+1}&=& \boldsymbol{\mu}_k+ \alpha \mathbf{h}(\mathbf{x}_k),\\
\label{equ_18341114}
\boldsymbol{\lambda}_{k+1}&=& \boldsymbol{\lambda}_k+ \alpha \mathbf{S}\mathbf{x}_k,
\end{eqnarray}
where $\alpha>0$ is chosen to ensure the stability of the algorithm. By reformulating the above iteration in terms of the $n$-dimensional components of the vectors $\mathbf{x}_k$ and $\boldsymbol{\lambda}_k$, and in terms of the scalar components of the vector $\boldsymbol{\mu}_k$, we obtain Algorithm $(\texttt{A}_1)$.

Since the matrix $\mathbf{S}$ is not full rank, we cannot apply directly existing results for regular (local) minimizers, such as Proposition 4.4.2, page 388, \cite{Ber99}. Still, for a local minimizer and Lagrange multipliers pair $(\mathbf{x}^*,\boldsymbol{\mu}^*,\boldsymbol{\lambda}^*)$, with $\boldsymbol{\lambda}^*\in \textmd{Range}(\mathbf{S})$, we show that if the initial values $\left(\mathbf{x}_0,\boldsymbol{\mu}_0,(\mathbf{I}-\mathbf{J})\boldsymbol{\lambda}_0\right)$ are close enough to  $(\mathbf{x}^*,\boldsymbol{\mu}^*,\boldsymbol{\lambda}^*)$, for a small enough step-size and under some conditions on (the Hessians of) the  functions $f_i(x)$ and $h_i(x)$, $i=1,\ldots,N$, the vectors $\mathbf{x}_k$ and $\boldsymbol{\mu}_k$ do indeed converge to $\mathbf{x}^*$ and $\boldsymbol{\mu}^*$, respectively.
However, although under the same conditions $\boldsymbol{\lambda}_k$ does converge, it cannot be guaranteed that it converges to the unique $\boldsymbol{\lambda}^*\in \textmd{Range}(\mathbf{S})$ but rather to a point in the set $\left\{\boldsymbol{\lambda}^*+\textmd{Null}\left(\mathbf{S}'\right)\right\}$.

The following theorem (whose proof can be found in the Appendix section)  addresses the local convergence properties of Algorithm $(A_1)$. It states that, under some assumptions on the functions $f_i(x)$ and $h_i(x)$, and provided  the initial values are close enough to a solution of the first order necessary conditions of $(P_2)$, and a small enough step-size $\alpha$ is used, the sequence  $\left\{\mathbf{x}_k,\boldsymbol{\mu}_k,\boldsymbol{\lambda}_k\right\}$ converges to the respective solution.

\begin{theorem}
\label{thm_17401116}
Let Assumptions \ref{assu_02091524} and \ref{assu_02091703} hold and let $\left(\mathbf{x}^*,\boldsymbol{\mu}^*,\boldsymbol{\lambda}^*\right)$ with $\boldsymbol{\lambda}^*\in \textmd{Range}(\mathbf{S})$, be a local minimizer-Lagrange multipliers pair of $(P_2)$. Assume also that  $\nabla_{\mathbf{x}\mathbf{x}}^2\boldsymbol{\mathcal{L}}
\left(\mathbf{x}^*,\boldsymbol{\mu}^*,\boldsymbol{\lambda}^*\right)$ is positive definite. Then there exists $\bar{\alpha}$, such that for all $\alpha \in (0,\bar{\alpha}]$, the set $\left(\mathbf{x}^*, \boldsymbol{\mu}^*,\boldsymbol{\lambda}^*+\textmd{Null}\left(\mathbf{S}'\right) \right)$ is an attractor of iteration (\ref{equ_18331114})-(\ref{equ_18341114}) and if the sequence $\left\{\mathbf{x}_k,\boldsymbol{\mu}_k,\boldsymbol{\lambda}_k\right\}$ converges to the set $\left(\mathbf{x}^*, \boldsymbol{\mu}^*,\boldsymbol{\lambda}^*+\textmd{Null}\left(\mathbf{S}'\right) \right)$, the rate of convergence of $\|\mathbf{x}_k-\mathbf{x}^*\|$, $\|\boldsymbol{\mu}_k-\boldsymbol{\mu}^*\|$ and $\left\|\boldsymbol{\lambda}_k-\left[\boldsymbol{\lambda}^*+\textmd{Null}\left(\mathbf{S}'\right)\right]\right\|$
is linear.
\end{theorem}

Let us know reformulate the above theorem so that the local convergence result can be applied to problem $(P_1)$.

\begin{corollary}
\label{cor_02111822}
Let Assumptions \ref{assu_02091524} and \ref{assu_02091703} hold and let $\left({x}^*,{\psi}^*\right)$  be a local minimizer-Lagrange multiplier pair of $(P_1)$. Assume also that  $\nabla^2f_i(x^*)+\psi_i^*\nabla^2h_i(x^*)$ are positive definite for all $i=1,\ldots,N$. Then there exists $\bar{\alpha}$, such that for all $\alpha \in (0,\bar{\alpha}]$, $\left({x}^*, {\psi}^*\right)$ is a point of attraction for  iteration (\ref{equ_02091554}) and (\ref{equ_02091555}), for all $i=1,\ldots,N$, and if the sequence $\left\{x_{i,k},\mu_{i,k}\right\}$ converges to $\left({x}^*,{\psi}^*\right)$, then the rate of convergence of $\|x_{i,k}-x^*\|$ and $\|\mu_{i,k}-\psi^*\|$ is linear.
\end{corollary}
\begin{proof}
By Proposition \ref{pro_14411114} we have that $\mathbf{x}^* = \mathds{1}\otimes x^*$ is a local minimizer of $(P_2)$ with corresponding Lagrange multipliers $\left(\boldsymbol{\mu}^*,\boldsymbol{\lambda}^*+
\textmd{Null}\left(\mathbf{S}'\right)\right)$, with $\boldsymbol{\lambda}^* \in \textmd{Range}\left(\mathbf{S}\right)$. In addition, by Proposition \ref{equ_02111222} we have that $\boldsymbol{\mu}^*=\psi^*$.
Using the definition of the Lagrangian function introduced in (\ref{equ_02111445}), we discover that
$$\nabla_{\mathbf{x}\mathbf{x}}^2\boldsymbol{\mathcal{L}}
\left(\mathbf{x}^*,\boldsymbol{\mu}^*,\boldsymbol{\lambda}^*\right)=
\textmd{diag}\left(\nabla^2f_i(x^*)+\psi_i^*\nabla^2h_i(x^*),i=1,\ldots,N\right).$$
But since we assumed that $\nabla^2f_i(x^*)+\psi_i^*\nabla^2h_i(x^*)\succ 0$ for all $i$, it follows that $\nabla_{\mathbf{x}\mathbf{x}}^2\boldsymbol{\mathcal{L}}
\left(\mathbf{x}^*,\boldsymbol{\mu}^*,\boldsymbol{\lambda}^*\right)\succ 0$ as well. Using Theorem \ref{thm_17401116}, the result follows.
\end{proof}

\begin{remark}
\label{rem_02111900}
In the previous corollary the matrices $\nabla^2f_i(x^*)+\psi_i^*\nabla^2h_i(x^*)$ were assumed to be positive definite for all $i=1,\ldots,N$. If we apply directly on $(P_1)$ results from the optimization literature (for instance Proposition 4.4.2, page 388, \cite{Ber99}) concerning convergence of  first-order methods used to compute local minimizers and their corresponding Lagrange multipliers, we only require $\sum_{i=1}^N\nabla^2f_i(x^*)+\psi_i^*\nabla^2h_i(x^*)$ to be positive definite, and not each element of the sum. Obviously the assumption in Corollary \ref{cor_02111822} does imply the latter, but is not necessary.$\Box$ 
\end{remark}

\subsection{Lagrangian methods  - Algorithm $(A_2)$}

In Theorem \ref{thm_17401116} we made the assumption that $\nabla_{\mathbf{x}\mathbf{x}}^2\boldsymbol{\mathcal{L}}
\left(\mathbf{x}^*,\boldsymbol{\mu}^*,\boldsymbol{\lambda}^*\right)$ is positive definite. We can relax this assumption by using an augmented version of the Lagrangian, obtaining an equivalent set of first order necessary conditions and applying again a first order numerical method to solve for the optimal solution.


The first order necessary conditions for $(P_2)$ with respect to the augmented Lagrangian $\boldsymbol{\mathcal{L}}_c(\mathbf{x},\boldsymbol{\mu},\boldsymbol{\lambda})$ are given by
\begin{eqnarray}
\label{equ_02181520}
\nabla \mathbf{F}(\mathbf{x})+
\nabla\mathbf{h}(\mathbf{x})\boldsymbol{\mu}+\mathbf{S}'\boldsymbol{\lambda}+
c\nabla\mathbf{h}(\mathbf{x})\mathbf{h}(\mathbf{x})+c\mathbf{L}\mathbf{x}
=0,\\
\label{equ_02181521}
\nabla \mathbf{h}(\mathbf{x})=0,\\
\label{equ_02181522}
\mathbf{S}\mathbf{x}=0.
\end{eqnarray}
A first order numerical method  that can be used  to solve the necessary conditions (\ref{equ_02181520})-(\ref{equ_02181522}) takes the form
\begin{eqnarray}
\nonumber
\mathbf{x}_{k+1} &=& \mathbf{x}_{k} - \alpha \left[\nabla \mathbf{F}(\mathbf{x}_k)+\nabla\mathbf{h}(\mathbf{x}_k)\boldsymbol{\mu}_k+\right.\\
\label{equ_02181523}
&+&\left.c\nabla\mathbf{h}(\mathbf{x}_k)\mathbf{h}(\mathbf{x}_k)+
\mathbf{S}'\boldsymbol{\lambda}_k+c\mathbf{L}\mathbf{x}_k\right],\\
\label{equ_02181524}
\boldsymbol{\mu}_{k+1}&=&\boldsymbol{\mu}_k+\alpha \mathbf{h}(\mathbf{x}_k),\\
\label{equ_02181525}
\boldsymbol{\lambda}_{k+1} &=& \boldsymbol{\lambda}_k+\alpha\mathbf{S}\mathbf{x}_k.
\end{eqnarray}
Re-writing the above iterations in terms of the components corresponding to each agent, we recover Algorithm $(\texttt{A}_2)$.

The following result addresses the local convergence properties of the iteration (\ref{equ_02181523})-(\ref{equ_02181525}). The proof of this result can be found in the Appendix section.
\begin{theorem}
\label{thm_02181546}
Let Assumptions \ref{assu_02091524} and \ref{assu_02091703} hold and let $\left(\mathbf{x}^*,\boldsymbol{\mu}^*,\boldsymbol{\lambda}^*\right)$ with $\boldsymbol{\lambda}^*\in \textmd{Range}(\mathbf{S})$, be a local minimizer-Lagrange multipliers pair of $(P_2)$. Assume also that  $\mathbf{x}'\nabla_{\mathbf{x}\mathbf{x}}^2\boldsymbol{\mathcal{L}}
\left(\mathbf{x}^*,\boldsymbol{\mu}^*,\boldsymbol{\lambda}^*\right)\mathbf{x}>0$ for all $\mathbf{x}\in \textmd{TC}(\mathbf{x}^*,\boldsymbol{\Omega})$. Then there exists $\bar{c}>0$ so that for all $c>\bar{c}$  we can find $\bar{\alpha}(c)$ such that for all $\alpha \in (0,\bar{\alpha}(c)]$, the set $\left(\mathbf{x}^*, \boldsymbol{\mu}^*,\boldsymbol{\lambda}^*+\textmd{Null}\left(\mathbf{S}'\right) \right)$ is an attractor of iteration (\ref{equ_02181523})-(\ref{equ_02181525}). In addition, if the sequence $\left\{\mathbf{x}_k,\boldsymbol{\mu}_k,\boldsymbol{\lambda}_k\right\}$ converges to the set $\left(\mathbf{x}^*, \boldsymbol{\mu}^*,\boldsymbol{\lambda}^*+\textmd{Null}\left(\mathbf{S}'\right) \right)$, the rate of convergence of $\|\mathbf{x}_k-\mathbf{x}^*\|$, $\|\boldsymbol{\mu}_k-\boldsymbol{\mu}^*\|$ and $\left\|\boldsymbol{\lambda}_k-\left[\boldsymbol{\lambda}^*+\textmd{Null}\left(\mathbf{S}'\right)\right]\right\|$
is linear.
\end{theorem}

The following corollary gives conditions that ensure local convergence to a local minimizer of $(P_1)$ for each agent following Algorithm $(A_2)$.
\begin{corollary}
\label{cor_02181555}
Let Assumptions \ref{assu_02091524} and \ref{assu_02091703} hold and let $\left({x}^*,{\psi}^*\right)$  be a local minimizer-Lagrange multiplier pair of $(P_1)$. Assume also that  $x'\sum_{i=1}^N\left[\nabla^2f_i(x^*)+\psi_i^*\nabla^2h_i(x^*)\right]x>0$ for all $x\in \textmd{TC}(x^*,\Omega)$. Then there exists $\bar{c}>0$ so that for all $c\geq \bar{c}$  we can find $\bar{\alpha}(c)$ such that for all $\alpha \in (0,\bar{\alpha}(c)]$, $\left({x}^*, {\psi}^*\right)$ is a point of attraction for  iteration (\ref{equ_02181536})-(\ref{equ_02181538}), for all $i=1,\ldots,N$. In addition, if the sequence $\left\{x_{i,k},\mu_{i,k}\right\}$ converges to $\left({x}^*,{\psi}^*\right)$, then the rate of convergence of $\|x_{i,k}-x^*\|$ and $\|\mu_{i,k}-\psi^*\|$ is linear.
\end{corollary}
\begin{IEEEproof}
By Proposition \ref{pro_14411114} we have that $\mathbf{x}^* = \mathds{1}\otimes x^*$ is a local minimizer of $(P_2)$ with corresponding Lagrange multipliers $\left(\boldsymbol{\mu}^*,\boldsymbol{\lambda}^*+
\textmd{Null}\left(\mathbf{S}'\right)\right)$, with $\boldsymbol{\lambda}^* \in \textmd{Range}\left(\mathbf{S}\right)$. In addition, by Proposition \ref{equ_02111222} we have that $\boldsymbol{\mu}^*=\psi^*$.
Using the definition of the Lagrangian function introduced in (\ref{equ_02111445}), we have
$$\nabla_{\mathbf{x}\mathbf{x}}^2\boldsymbol{\mathcal{L}}
\left(\mathbf{x}^*,\boldsymbol{\mu}^*,\boldsymbol{\lambda}^*\right)=
\textmd{diag}\left(\nabla^2f_i(x^*)+\psi_i^*\nabla^2h_i(x^*),i=1,\ldots,N\right).$$
In Proposition \ref{pro_02181835} we showed that
$$\textmd{TC}(\mathbf{x}^*,\boldsymbol{\Omega}) = \left\{\mathds{1}\otimes z \ |\ z\in \textmd{TC}(x^*,\Omega)\right\},$$
and therefore the assumption $x'\sum_{i=1}^N\left[\nabla^2f_i(x^*)+\psi_i^*\nabla^2h_i(x^*)\right]x>0$ for all $x\in \textmd{TC}(x^*,\Omega)$ is equivalent to
$\mathbf{x}'\nabla_{\mathbf{x}\mathbf{x}}^2\boldsymbol{\mathcal{L}}
\left(\mathbf{x}^*,\boldsymbol{\mu}^*,\boldsymbol{\lambda}^*\right)\mathbf{x}>0$ for all $\mathbf{x}\in \textmd{TC}(\mathbf{x}^*,\boldsymbol{\Omega})$. Using Theorem \ref{thm_02181546}, the result follows.
\end{IEEEproof}
\begin{remark}
\label{rem_02111901}
In the previous corollary we made the assumption that $x'\sum_{i=1}^N\left[\nabla^2f_i(x^*)+\psi_i^*\nabla^2h_i(x^*)\right]x>0$ for all $x\in \textmd{TC}(x^*,\Omega)$. It turns out that the same assumption appears if we apply directly on $(P_1)$ standard results from the optimization literature concerning convergence of first-order methods used to solve the first order necessary conditions involving the augmented Lagrangian. In other words the assumption is identical when solving $(P_1)$ in a centralized manner.$\Box$ 
\end{remark}

\section{Convergence analysis of the distributed algorithms based on the method of multipliers}
\label{sec_12011444}
In this section we prove a set of results that will be used to give conditions under which \texttt{Algorithm} $(\texttt{A}_3)$ converges to a local minimizer of $(\texttt{P}_1)$. The results are modifications of standard results concerning the method of multipliers (see for example Section 2.2, \cite{Bertsekas1982}). In the standard case, a regularity assumption on the minimizers is used to prove convergence of the method of multipliers. In our setup, this is not the case anymore and therefore, the standard results need to be modified accordingly.

Considering the notations introduced so far, Algorithm $(\texttt{A}_3)$ can be compactly written as
{\small \begin{eqnarray}
\label{equ_03251133}
\mathbf{x}_{k} &=& \arg \min_{\bf{x}} \boldsymbol{\mathcal{L}}_{c_k} (\mathbf{x},\boldsymbol{\mu}_k,\boldsymbol{\lambda}_k),\ \mathbf{x}_{0} = \mathbf{x}^0,\\
\label{equ_03251134}
\boldsymbol{\mu}_{k+1} &=&\boldsymbol{\mu}_{k}+c_k \mathbf{h}(\mathbf{x}_{k}),\ \boldsymbol{\mu}_{0}=\boldsymbol{\mu}^0,\\
\label{equ_03251135}
\boldsymbol{\lambda}_{k+1} &=& \boldsymbol{\lambda}_{k}+c_k \mathbf{S}\mathbf{x}_k,\ \boldsymbol{\lambda}_{0}=\boldsymbol{\lambda}^0.
\end{eqnarray}}
where the unconstrained optimization problem (\ref{equ_03251133}) is solved using the iteration
$$\mathbf{x}_{k,\tau+1} = \mathbf{x}_{k,\tau} - \alpha_{\tau}\left\{ \nabla \mathbf{F}(\mathbf{x}_{k,\tau})+
\nabla\mathbf{h}(\mathbf{x}_{k,\tau})\boldsymbol{\mu}_k+
\mathbf{S}'\boldsymbol{\lambda}_k+c_k\nabla\mathbf{h}(\mathbf{x}_{k,\tau})\mathbf{h}(\mathbf{x}_{k,\tau})+c_k \mathbf{L}\mathbf{x}_{k,\tau}\right\}, \mathbf{x}_{k,0} = \mathbf{x}_{k-1}.$$

Therefore, iterations (\ref{equ_02091554})-(\ref{equ_02091556}) reflect the method of multipliers applied to Problem $(P_2)$, while the iteration (\ref{equ_03121344}) is a gradient descent method for finding a local minimizer of the unconstrained optimization problem $\min_{\mathbf{x}}\boldsymbol{\mathcal{L}}_c (\mathbf{x},\boldsymbol{\mu}_k,\boldsymbol{\lambda}_k)$.

\subsection{Supporting results}
 Before stating the conditions under which convergence can be achieved, let us first start by enumerating a set of supporting results used for the convergence analysis. Their complete proofs can be found in the Appendix section. To simplify notations, let us group the Lagrange multipliers in one vector, that is, $\boldsymbol{\eta}'=[\boldsymbol{\mu}',\boldsymbol{\lambda}']$. When we mention a Lagrange multiplier vector $\boldsymbol{\eta}^*$ corresponding to a local minimizer $\mathbf{x}^*$, we will understand that its subcomponent $\boldsymbol{\lambda}^*$ is the unique vector in $\textmd{Range}(\mathbf{S})$. In addition, let us group the equality constraints functions of $(P_2)$ into one vector-valued function, that is, $\tilde{\mathbf{h}}(\mathbf{x})' = \left(\mathbf{h}(\mathbf{x})',\mathbf{x}'\mathbf{S}'\right)$.

The convergence of the distributed algorithm derived from method of multipliers is based on the following  result, which is an adaptation of the Proposition 2.4, page 108 of \cite{Bertsekas1982} so that it fits our setup.
\begin{proposition}
\label{pro_03051249}
Let Assumptions \ref{assu_02091524} and \ref{assu_02091703} hold, let $(\mathbf{x}^*,\boldsymbol{\eta}^*)$, be a local minimizer-Lagrange multipliers pair of $(P_2)$, and assume that $\mathbf{z}'\nabla_{\mathbf{x}\mathbf{x}}^2\boldsymbol{\mathcal{L}}_{0}(\mathbf{x}^*,\boldsymbol{\eta}^*)\mathbf{z}>0$ for all $\mathbf{z}\in \textmd{TC}(\mathbf{x}^*,\boldsymbol{\Omega})$. In addition let $\bar{c}$ be a positive scalar such that $\nabla_{\mathbf{x}\mathbf{x}}^2\boldsymbol{\mathcal{L}}_{\bar{c}}(\mathbf{x}^*,\boldsymbol{\eta}^*)\succ 0$. There exists positive scalars $c_{\textmd{max}}$, $\delta$, $\varepsilon$ and $M$ such that:
\begin{enumerate}[(a)]
\item For all $\left(\boldsymbol{\eta},c\right)$ in the set $D\subset \mathds{R}^{N+n\bar{N}+1}$ defined by
{\small \begin{equation}
\label{equ_03051258}
D=\left\{\left(\boldsymbol{\eta},c\right) \ |\ \|\mathbf{T}\boldsymbol{\eta}-\boldsymbol{\eta}^*\|< c\delta,\ \bar{c}\leq c \leq c_{\textmd{max}} \ \right\},
\end{equation}}
where
{\small $$\mathbf{T}=\left[\begin{array}{cc}
\mathbf{I} & \mathbf{0}\\
\mathbf{0} & \mathbf{I}-\mathbf{J}
\end{array}
\right],$$}
with $\mathbf{J}$ the orthogonal projection operator on $\textmd{Null}\left(\mathbf{S}'\right)$, the problem
{\small \begin{eqnarray}
\label{equ_03051312}
\min & \boldsymbol{\mathcal{L}}_c(\mathbf{x},\boldsymbol{\eta})\\
\nonumber
\textmd{subject to } & \mathbf{x}\in \mathcal{S}(\mathbf{x}^*;\varepsilon)
\end{eqnarray}}
has a unique solution denoted by $\mathbf{x}(\boldsymbol{\eta},c)$. The function $\mathbf{x}(\cdot,\cdot)$ is continuously differentiable in the interior of $D$, and for all $\left(\boldsymbol{\eta},c\right)\in D$, we have
{\small \begin{equation}
\label{equ_03051319}
\|\mathbf{x}(\boldsymbol{\eta},c)-\mathbf{x^*}\|\leq M \|\boldsymbol{\eta}-\boldsymbol{\eta}^*\|/c.
\end{equation}}
\item For all $\left(\boldsymbol{\eta},c\right)\in D$, we have
{\small \begin{equation}
\label{equ_03051321}
\|\tilde{\boldsymbol{\eta}}(\boldsymbol{\eta},c)-\boldsymbol{\eta}^*\|\leq M \|\boldsymbol{\eta}-\boldsymbol{\eta}^*\|/c,
\end{equation}}
where
{\small \begin{equation}
\label{equ_03051431}
\tilde{\boldsymbol{\eta}}(\boldsymbol{\eta},c) = \mathbf{T}\boldsymbol{\eta}+c \tilde{\mathbf{h}}\left(\mathbf{x}(\boldsymbol{\eta},c)\right).
\end{equation}}
\item For all $\left(\boldsymbol{\eta},c\right)\in D$, the matrix $\nabla_{\mathbf{x}\mathbf{x}}^2\boldsymbol{\mathcal{L}}_c[\mathbf{x}(\boldsymbol{\eta},c),\boldsymbol{\eta}]$ is positive definite.
\end{enumerate}
\end{proposition}
The proof of this proposition can be found in the Appendix section. At this point we would like only the emphasize what is different compared with the original formulation.
First recall that the assumption on the Hessian $\nabla_{\mathbf{x}\mathbf{x}}^2\boldsymbol{\mathcal{L}}_{0}(\mathbf{x}^*,\boldsymbol{\eta}^*)$ simply means that $\mathbf{x}^*$ is a strictly local minimizer. Generally speaking, the proof of this result follows the same lines as the proof of Proposition 2.4, \cite{Bertsekas1982}. However, since the local minimizer $\mathbf{x}^*$ is not regular, we made some changes in the statement of the proposition compared to the original result, and consequently we need to adapt the proof accordingly. Compared to Proposition 2.4, page 108 of \cite{Bertsekas1982}, our results has three main differences. The first difference consists of imposing an upper bound on $c$, namely $c_{\textmd{max}}$.  The reason we introduced $c_{\textmd{max}}$ is to ensure that a certain Jacobian matrix that depends on $c$ (explicitly defined in the proof) is invertible.  The scalar $c_{\textmd{max}}$ however can be made arbitrarily large. The price paid for this change is the prevention of reaching a theoretical superlinear rate of convergence for Algorithm $(\texttt{A}_3)$.

The second difference is the introduction of operator $\mathbf{T}$. This operator acting on $\boldsymbol{\eta}$ ensures that $\tilde{\boldsymbol{\lambda}}(\boldsymbol{\eta},c)\in \textmd{Range}(\mathbf{S})$ for all $(\boldsymbol{\eta},c)\in D$, where $\tilde{\boldsymbol{\eta}}=(\tilde{\boldsymbol{\mu}},\tilde{\boldsymbol{\lambda}})$. In defining the set $D$, the matrix $\mathbf{T}$ induces a neighborhood around $\boldsymbol{\eta}^*$ were only points $\boldsymbol{\eta} = (\boldsymbol{\mu},\boldsymbol{\lambda})$, with $\boldsymbol{\lambda}\in \textmd{Range}(\mathbf{S})$ are considered. In this neighborhood, $\boldsymbol{\eta}^* = (\boldsymbol{\mu}^*,\boldsymbol{\lambda}^*)$ is a unique Lagrange multiplier vector corresponding to the local minimizer $\mathbf{x}^*$.

The third difference is the definition of $\tilde{\boldsymbol{\eta}}$ in (\ref{equ_03051431}). Compared to the original statement\footnote{$\boldsymbol{\eta}$ corresponds to $\lambda$ in Proposition 2.4, \cite{Bertsekas1982}.} of Proposition 2.4, \cite{Bertsekas1982}, we introduce the operator $\mathbf{T}$ that multiplies $\boldsymbol{\eta}$, to deal with the fact  as $\mathbf{x}^*$ is not regular local minimizer. As a consequence, we will have that $\tilde{\boldsymbol{\lambda}}(\boldsymbol{\eta},c)\in \textmd{Range}(\mathbf{S})$ for all $(\boldsymbol{\eta},c)\in D$, where $\tilde{\boldsymbol{\eta}} = (\tilde{\boldsymbol{\mu}},\tilde{\boldsymbol{\lambda}})$.$\Box$

Given a minimizer-Lagrange multiplier pair $(\mathbf{x}^*,\boldsymbol{\eta}^*)$ of $(P_2)$, let us define the following matrix
{\small \begin{equation}
\label{equ_03061931}
{\mathbf{M}} = \left\{\nabla \tilde{\mathbf{h}}(\mathbf{x}^*)'\left[\nabla_{\mathbf{x}\mathbf{x}}^2{\boldsymbol{\mathcal{L}}}_c\left(\mathbf{x}^*,\boldsymbol{\eta}^*\right)\right]^{-1}\nabla \tilde{\mathbf{h}}(\mathbf{x}^*)\right\}^{-1}-c\mathbf{I}
\end{equation}}
for any $c$ for which $\nabla_{\mathbf{x}\mathbf{x}}^2{\boldsymbol{\mathcal{L}}}_c\left(\mathbf{x}^*,\boldsymbol{\eta}^*\right)$ is invertible.
In addition, it can be shown that if $\left[\nabla_{\mathbf{x}\mathbf{x}}^2{\boldsymbol{\mathcal{L}}}_0\left(\mathbf{x}^*,\boldsymbol{\eta}^*\right)\right]^{-1}$ exist, then
{\small \begin{equation}
\label{equ_03061939}
{\mathbf{M}} = \left\{\nabla \tilde{\mathbf{h}}(\mathbf{x}^*)'\left[\nabla_{\mathbf{x}\mathbf{x}}^2{\boldsymbol{\mathcal{L}}}_0\left(\mathbf{x}^*,\boldsymbol{\eta}^*\right)\right]^{-1}\nabla \tilde{\mathbf{h}}(\mathbf{x}^*)\right\}^{-1},
\end{equation}}
respectively.

\subsection{Method of multipliers - Algorithm ($\texttt{A}_3$)}
The following Theorem is the equivalent of Proposition 2.7 of \cite{Bertsekas1982} and if formulated keeping in mind the distributed setup of the problem. Our result does not include the case of superlinear rate of convergence since we upper bound the scalar $c$ and its statement is  adapted so that it fits to the characteristics of Problem $(P_2)$.
\begin{theorem}
\label{pro_03061521}
Let Assumptions \ref{assu_02091524} and \ref{assu_02091703} hold, let $(\mathbf{x}^*,\boldsymbol{\eta}^*)$, with $\boldsymbol{\eta}^*=(\boldsymbol{\mu}^*,\boldsymbol{\lambda}^*)$ be a local minimizer-Lagrange multipliers pair of $(P_2)$ that satisfies
$\mathbf{z}'\nabla_{\mathbf{x}\mathbf{x}}^2{\boldsymbol{\mathcal{L}}}_0\left(\mathbf{x}^*,\boldsymbol{\eta}^*\right)\mathbf{z}>0$ for all $\mathbf{z}\in \textmd{TC}(\mathbf{x}^*,\boldsymbol{\Omega})$. In addition, let $\bar{c}$, $\delta$ and $c_{\textmd{max}}$ be as in Proposition \ref{pro_03051249} with $\bar{c}>\max\{-2{\mathbf{e}}_1,\ldots,-2{\mathbf{e}}_{nN}\}$,
where $\mathbf{e}_1,\ldots,\mathbf{e}_{nN}$ are the eigenvalues of $\mathbf{M}$ defined in (\ref{equ_03061931}). Then there exists $\delta_1$ with $0<\delta_1\leq \delta$ such that if $\{c_k\}$ and $\boldsymbol{\eta}_0$ satisfy
{\small \begin{equation}
\label{equ_03062005}
\|\mathbf{T}\boldsymbol{\eta}_0-\boldsymbol{\eta}^*\|<\delta_1c_0,\  \bar{c}\leq c_k\leq c_{k+1}\leq c_{\textmd{max}}\ \forall k,
\end{equation}}
then for the sequence $\{\boldsymbol{\eta}_k\}$ generated by
{\small \begin{equation}
\label{equ_03062006}
\boldsymbol{\eta}_{k+1} = \mathbf{T}\boldsymbol{\eta}_k+c_k\tilde{\mathbf{h}}\left(\mathbf{x}(\boldsymbol{\eta}_k,c_k)\right),
\end{equation}}
we have that $\|\boldsymbol{\eta}_k-\boldsymbol{\eta}^*\|$ and $\|\mathbf{x}(\boldsymbol{\eta}_k,c_k)-\mathbf{x}^*\|$ converge to zero. Furthermore if  $\boldsymbol{\eta}_{k}\neq \boldsymbol{\eta}^*$ for all $k$, there holds
{\small \begin{equation}
\label{equ_03062013}
\lim \sup_{k\rightarrow \infty} \frac{\|\boldsymbol{\eta}_{k+1}-\boldsymbol{\eta}^*\|}{\|\boldsymbol{\eta}_{k}-\boldsymbol{\eta}^*\|}\leq \max_{i=1\ldots nN}\left|\frac{{\mathbf{e}_i}}{{\mathbf{e}_i}+c_{\textmd{max}}}\right|. \Box
\end{equation}}
\end{theorem}

In the following we show that under some conditions, \texttt{Algorithm} $(\texttt{A}_3)$ ensures the convergence of each $x_{i,k}$ to $x^*$ and the convergence of $\mu_{i,k}$ to $\mu_i^*$, where $\boldsymbol{\mu}^* = (\mu_i^*)$. The algorithm does not guarantee convergence of $\lambda_{ij,k}$ to $\lambda_{ij}^*$, where $\boldsymbol{\lambda}^* = (\lambda_{ij}^*)$, for $j\in \mathcal{N}_i$ and $i=1,\ldots,N$. In fact, if convergence is achieved, $\boldsymbol{\lambda}_k$ converges to the set $\boldsymbol{\lambda}^*+\textmd{Null}(\mathbf{S}')$.

\begin{corollary}
\label{cor_03122121}
Let Assumptions \ref{assu_02091524} and \ref{assu_02091703} hold, let $({x}^*,\psi^*)$ be a local minimizer-Lagrange multipliers pair of $(P_1)$. In addition, let $\mathbf{x}^* = \mathds{1}\otimes x^*$ be a local minimizer of $(P_2)$ (as stated by Proposition \ref{pro_14411114}), with unique Lagrange multipliers   $(\boldsymbol{\mu}^*,\boldsymbol{\lambda}^*)$ and $\boldsymbol{\lambda}^* \in \textmd{Range}(\mathbf{S})$. Assume also that
${z}'\sum_{i=1}^N\left(\nabla^2f_i(x^*)+\psi_i^*\nabla^2 h_i(x^*)\right){z}>0$ for all ${z}\in \textmd{TC}({x}^*,{\Omega})$. There exist scalars $c_{\textmd{max}}$, $\delta$, $\delta_1$ with   $0<\delta_{1}\leq \delta$ and $c_{\textmd{max}}$ satisfying $\bar{c}>\max\{-2{\mathbf{e}}_1,\ldots,-2{\mathbf{e}}_{nN}\}$ where $\mathbf{e}_1,\ldots,\mathbf{e}_{nN}$ are the eigenvalues of $\mathbf{M}$ defined in (\ref{equ_03061931}) such that if the sequence $\{c_k\}$, $\boldsymbol{\mu}_0$ and $\boldsymbol{\lambda}_0$ satisfy
{\small \begin{eqnarray}
\label{equ_03071833}
&\left(\|{\boldsymbol{\mu}}_0-{\psi}^*\|^2 + \| (\mathbf{I}-\mathbf{J}) {\boldsymbol{\lambda}}_0 -\boldsymbol{\lambda}^* \|^2\right)^{1/2} <\delta_{1}c_0,\\
& \bar{c}\leq c_k\leq c_{k+1}\leq c_{\textmd{max}},
\end{eqnarray}}
then for the sequences $\{x_{i,k}\}$ and $\{{\mu}_{i,k}\}$ generated by the iteration (\ref{equ_03251133})-(\ref{equ_03251135}), we have that $\|{x}_{i,k}- {x}^*\|$ and $\|{\mu}_{i,k}- {\psi}_i^*\|$ converge to zero. Furthermore if ${\mu}_{i,k}\neq {\psi}_i^*$ for all $k$, the rate of convergence of $\{{\mu}_{i,k}\}$ is linear.
\end{corollary}
\begin{IEEEproof}
By Proposition \ref{equ_02111222} we have that $\boldsymbol{\mu}^*=\psi^*$.
Using the definition of the Lagrangian function introduced in (\ref{equ_02111445}), we have $\nabla_{\mathbf{x}\mathbf{x}}^2\boldsymbol{\mathcal{L}}
\left(\mathbf{x}^*,\boldsymbol{\mu}^*,\boldsymbol{\lambda}^*\right)=
\textmd{diag}\left(\nabla^2f_i(x^*)+\mu_i^*\nabla^2h_i(x^*),i=1,\ldots,N\right).$
In Proposition \ref{pro_02181835} we showed that $\textmd{TC}(\mathbf{x}^*,\boldsymbol{\Omega}) = \left\{\mathds{1}\otimes h \ |\ h\in \textmd{TC}(x^*,\Omega)\right\},$
and therefore the assumption $z'\sum_{i=1}^N\left[\nabla^2f_i(x^*)+\mu_i^*\nabla^2h_i(x^*)\right]z>0$ for all $x\in \textmd{TC}(x^*,\Omega)$ is equivalent to $\mathbf{x}'\nabla_{\mathbf{x}\mathbf{x}}^2\boldsymbol{\mathcal{L}}_0
\left(\mathbf{x}^*,\boldsymbol{\mu}^*,\boldsymbol{\lambda}^*\right)\mathbf{x}> 0\ \forall \mathbf{x}\in \textmd{TC}(\mathbf{x}^*,\boldsymbol{\Omega}).$
Pick $\bar{c}$ such that $\nabla_{\mathbf{x}\mathbf{x}}^2\boldsymbol{\mathcal{L}}_{\bar{c}}
\left(\mathbf{x}^*,\boldsymbol{\mu}^*,\boldsymbol{\lambda}^*\right)\succ 0$ and $\bar{c}>\max\{-2{\mathbf{e}}_1,\ldots,-2{\mathbf{e}}_{nN}\}$. Let $c_{\textmd{max}}$, $\delta$ and $\delta_1$ as in Proposition \ref{pro_03051249}. As before, we make the following transformation: $\bar{\boldsymbol{\lambda}}_k = (\mathbf{I}-\mathbf{J})\boldsymbol{\lambda}_k$, where $\mathbf{J}$ is the projection operator on $\textmd{Null}(\mathbf{S}')$. Making the observation that $\mathbf{S}' = \mathbf{S}'(\mathbf{I}-\mathbf{J})$ and that $(\mathbf{I}-\mathbf{J})\bar{\boldsymbol{\lambda}}_k=(\mathbf{I}-\mathbf{J})\boldsymbol{\lambda}_k$, iterations (\ref{equ_03251133})-(\ref{equ_03251135}) become
{\small \begin{eqnarray}
\label{equ_03251148}
\mathbf{x}_{k} &=& \arg \min_{\bf{x}} \boldsymbol{\mathcal{L}}_{c_k} (\mathbf{x},\boldsymbol{\mu}_k,\bar{\boldsymbol{\lambda}}_k),\ \mathbf{x}_{0} = \mathbf{x}^0,\\
\label{equ_03251149}
\boldsymbol{\mu}_{k+1} &=&\boldsymbol{\mu}_{k}+c_k \mathbf{h}(\mathbf{x}_{k}),\ \boldsymbol{\mu}_{0}=\boldsymbol{\mu}^0,\\
\label{equ_03251150}
\bar{\boldsymbol{\lambda}}_{k+1} &=& (\mathbf{I}-\mathbf{J})\bar{\boldsymbol{\lambda}}_{k}+c_k \mathbf{S}\mathbf{x}_k,\ \bar{\boldsymbol{\lambda}}_{0}=\bar{\boldsymbol{\lambda}}^0.
\end{eqnarray}}
that are exactly the iterations found in Theorem \ref{pro_03061521}. All assumptions of Theorem \ref{pro_03061521} are satisfied and the result follows.
\end{IEEEproof}
\begin{remark}
At each step of Algorithm $(\texttt{A}_3)$, we use iteration (\ref{equ_03121344}) to obtain the solution of (\ref{equ_02091554}), and therefore the convergence is dependent on solving
{\small \begin{eqnarray}
\nonumber
\min & \boldsymbol{\mathcal{L}}_c(\mathbf{x},\boldsymbol{\eta}_k)\\
\nonumber
\textmd{subject to } & \mathbf{x}\in \mathcal{S}(\mathbf{x}^*;\varepsilon).
\end{eqnarray}}
The solution of the above problem is well defined if it is "close enough" to the local minimizer $\mathbf{x}^*$. However, the unconstrained optimization problem may have multiple local minimizers. Thus for the algorithm to converge to the correct solution, $\mathbf{x}_k$ must remain in a neighborhood of the same local minimizer, at least after some time instant $k$.
Practice  showed that using $\mathbf{x}_k$ as starting point in (\ref{equ_03121344}) to compute $\mathbf{x}_{k+1}$ tends to ensure that the solutions of the unconstrained optimization problems remain in a neighborhood of the same local minimizer. In addition to starting closed enough from $\textbf{x}^*$, appropriate step-sizes $\alpha_{\tau}$ must be used so that (\ref{equ_03121344}) converges. A "sufficiently small" constant sequence ($\alpha_{\tau}=\alpha$) or a slowly diminishing sequence ($\alpha_{\tau}\rightarrow 0, \sum_{\tau}\alpha_{\tau} = \infty$) can be chosen. Conditions on the stepsize sequence that ensure convergence can be found in \cite{Ber99} (Propositions 1.2.3 and 1.2.4).$\Box$
\end{remark}
\begin{remark}
\label{rem_02111900}
Corollary \ref{cor_03122121} shows that we can use the method of multipliers to compute a local minimizer for the Problem $(P_1)$. Note the change in the condition the initial values $\boldsymbol{\eta}_0$ must satisfy, compared to the original result, namely the projection of $\boldsymbol{\lambda}_0$ on Range(\textbf{S}). This change was necessary as a result of the lack of regularity of the local minimizer; lack of regularity that also prevented us from showing that Algorithm $(A_1)$ can achieve a (theoretical) superlinear rate of convergence due to the upper-bound imposed on  the sequence $\{c_k\}$.$\Box$
\end{remark}

Although we assumed the sequence $\{c_k\}$ to be globally known by all agents, there are simple strategies to choosing such sequences that do not require significant initial communication overhead. For example, each agent can compute $c_k$  according to the scheme $c_{k+1} = \max\{\beta c_k, c_{\textmd{max}}\}$ for some scalar $\beta>1$, and where only the initial value  $c_0$, the upper-bound $c_{\textmd{max}}$ and the scalar $\beta$ must be known by all agents. 
 As seen earlier, the method of multipliers involves solving at each time instant the unconstrained optimization problem (\ref{equ_02091554}). To solve this step, we use a gradient method (iteration (\ref{equ_03121344})), since, due to the nature of the cost function, it can be implemented in a  distributed manner. In practice, the step (\ref{equ_02091554}) is not solved exactly, and usually the iteration (\ref{equ_03121344}) is stopped as some stopping criterion is satisfied. Section 2.5 of \cite{Bertsekas1982} introduces the ``asymptotically exact minimization in methods of multipliers'', which basically shows several strategies for solving approximately  the step  (\ref{equ_02091554}), and still obtain convergence to a local minimizer.

\bibliographystyle{plain}
\bibliography{references}

\appendix

\subsection{Proofs of the supporting results for the distributed algorithm based on the Lagrangian methods}
\label{sec_02191015}
This section includes the missing proofs of the supporting results in Section \ref{sec_12011443} plus the auxiliary results necessary to achieve this goal. We start with a well known result on the properties of the tangent cone to the constraint set at a local minimizer of $(\texttt{P}_1)$.
\begin{proposition}
\label{pro_02110939}
Let Assumptions \ref{assu_02091524}-(a) and \ref{assu_02091703} hold, let $x^*$ be a local minimizer of $(\texttt{P}_1)$ and let $\Omega$ denote the constraint set, that is, $\Omega = \{x\ |\ h(x)=0\}$. Then the tangent cone to $\Omega$ at $x^*$ is given by
$\textmd{TC}(x^*,\Omega) = \textmd{Null}\left(\nabla h(x^*)'\right),$
where $\nabla h\left(x^*\right)\triangleq \left[\nabla h_1\left(x^*\right),\nabla h_2\left(x^*\right),\ldots,\nabla h_N\left(x^*\right)\right]$.
\end{proposition}
Let $\mathbf{x}^*=\mathds{1}\otimes x^*$ denote  a  local minimizer of $(\texttt{P}_2)$ and let $\nabla \mathbf{h}(\mathbf{x}^*)$ denote the matrix
$\nabla \mathbf{h}(\mathbf{x}^*) \triangleq \left[\nabla \mathbf{h}_1(\mathbf{x}^*),\nabla \mathbf{h}_2(\mathbf{x}^*),\ldots,\nabla \mathbf{h}_N(\mathbf{x}^*)\right].$
The vectors $\nabla\mathbf{h}_i(\mathbf{x}^*)$ are the gradients of the functions $\mathbf{h}_i(\mathbf{x})$  at $\mathbf{x}^*$ with a structure given by
{\small
 \begin{equation}
 \label{equ_02091847}
\nabla \mathbf{h}_i(\mathbf{x}^*)'=\left[\underbrace{0,\ldots,0}_{n \textmd{ zeros}},\ldots,\underbrace{0,\ldots,0}_{n \textmd{ zeros}},\underbrace{\nabla h_i(x^*)'}_{i^{\textmd{th}}\ \textmd{component}},\underbrace{0,\ldots,0}_{n \textmd{ zeros}},\ldots,\underbrace{0,\ldots,0}_{n \textmd{ zeros}}\right],
 \end{equation}}
 as per definition of the function $\mathbf{h}_i(\mathbf{x})$.

 The second result of this section is concerned with the nullspace of the matrix $\left[\nabla \mathbf{h}(\mathbf{x}^*),\mathbf{S}'\right]$, which will be used to characterize the tangent cone at a local minimizer of $(\texttt{P}_2)$.
\begin{proposition}
\label{pro_02091833}
Let Assumptions \ref{assu_02091524} and \ref{assu_02091703} hold. The nullspace of the matrix $\left[\nabla \mathbf{h}(\mathbf{x}^*),\mathbf{S}'\right]$ is given by
$\textmd{Null}\left(\left[\nabla \mathbf{h}(\mathbf{x}^*),\mathbf{S}'\right]\right)=
\left\{\left(\mathbf{0}',\mathbf{v}'\right)' \ |\ \mathbf{v}\in \textmd{Null}\left(\mathbf{S}'\right)\right\}.$
\end{proposition}
\begin{IEEEproof}
Let $\mathbf{u}\in \mathds{R}^N$ and $\mathbf{v}\in \mathds{R}^{nN}$ be two vectors. To characterize the nullspace of $\left[\nabla \mathbf{h}(\mathbf{x}^*),\mathbf{S}'\right]$ we need to check for what values of $\mathbf{u}$ and $\mathbf{v}$ the equation
\begin{equation}
\label{equ_02091845}
\nabla \mathbf{h}(\mathbf{x}^*) \mathbf{u}+ \mathbf{S}' \mathbf{v}=\mathbf{0}
\end{equation}
is satisfied. Using the definition of $\nabla \mathbf{h}_i(\mathbf{x}^*)$ shown in (\ref{equ_02091847}), equation (\ref{equ_02091845}) can be equivalently written as
$$\nabla {h}_i({x}^*)\mathbf{u}_i+\sum_{j\in \mathcal{N}_i}\left(s_{ij}v_{ij}-s_{ji}{v}_{ji}\right)=0, \ i=1,\ldots,N,$$
where $\mathbf{u} = (\mathbf{u}_i)$ with $\mathbf{u}_i \in \mathds{R}$ and $i=1,\ldots,N$, and $\mathbf{v} = (\mathbf{v}_i)$ with $\mathbf{v}_i\in \mathds{R}^{n|\mathcal{N}_i|}$ and $\mathbf{v}_i = (v_{ij})$ with $v_{ij}\in \mathds{R}^n$ and $j\in \mathcal{N}_i$.
Summing the above equations over $i$ we obtain that
$$\sum_{i=1}^N\mathbf{u}_i\nabla h_i(x^*)=0,$$
and since $\nabla h(x^*)$ is assumed full rank we must have that $\mathbf{u}=0$ and the result follows.
\end{IEEEproof}

We now have all the machinery necessary to characterize the tangent cone at a local minimizer of $(\texttt{P}_2)$.

\begin{IEEEproof}[Proof of Proposition \ref{pro_02181835}]
All we have to show is that any vector in $\textmd{Null}\left(\left[\nabla \mathbf{h}(\mathbf{x}^*),\mathbf{S}'\right]'\right)$ belongs to $\textmd{TC}\left(\mathbf{x}^*,\boldsymbol{\Omega}\right)$ as well, since it is well known that (the closure of the convex hull of) $\textmd{TC}\left(\mathbf{x}^*,\boldsymbol{\Omega}\right)$ is included in $\textmd{Null}\left(\left[\nabla \mathbf{h}(\mathbf{x}^*),\mathbf{S}'\right]'\right)$.
Let $\mathbf{u}$ be a vector in $\textmd{Null}\left(\left[\nabla \mathbf{h}(\mathbf{x}^*),\mathbf{S}'\right]'\right)$ and therefore it must satisfy
\begin{equation}
\label{equ_02111112}
\nabla \mathbf{h}(\mathbf{x}^*)'\mathbf{u}=0 \textmd{ and } \mathbf{S}\mathbf{u}=0.
\end{equation}
From the second equation of (\ref{equ_02111112}), $\mathbf{u}$ must be of the form
$\mathbf{u}=\mathds{1}\otimes u$, for some $u\in \mathds{R}^n$. From the first equation of (\ref{equ_02111112}), using the definition of $\nabla \mathbf{h}_i(\mathbf{x}^*)$ in (\ref{equ_02091847}) together with the particular structure of $\mathbf{u}$, we obtain that
$$\nabla h_i(x^*)'u=0 \ \forall i=1,\ldots,N,$$
or equivalently
\begin{equation*}
\label{equ_02111120}
u\in \textmd{Null}\left(\nabla h(x^*)'\right).
\end{equation*}
We need to show that a vector $\mathbf{u}=\mathds{1}\otimes u$, with $u\in \textmd{Null}\left(\nabla h(x^*)'\right)$ belongs to $\textmd{TC}\left(\mathbf{x}^*,\boldsymbol{\Omega}\right)$. More explicitly, using the definition of the tangent cone, we must find a function $\mathbf{o}:\mathds{R}\rightarrow \mathds{R}^{nN}$, with $\lim_{t\rightarrow 0, t>0}\frac{\mathbf{o}(t)}{t}=0$, so that $$\mathbf{x}^* + t \mathbf{u}+\mathbf{o}(t)\in \boldsymbol{\Omega}\ \forall t>0.$$
Choosing $\mathbf{o}(t) = \mathds{1}_N\otimes o(t)$, where $o:\mathds{R}\rightarrow \mathds{R}^n$ is a function so that $\lim_{t\rightarrow 0, t>0}\frac{{o}(t)}{t}=0$, we note that
$$\mathbf{g}\left(\mathbf{x}^* + t \mathbf{u}+\mathbf{o}(t)\right)=0\ \forall t>0,$$
and therefore, all we are left to do is to check that
\begin{equation}
\label{equ_02111129}
\mathbf{h}\left(\mathbf{x}^* + t \mathbf{u}+\mathbf{o}(t)\right)=0\ \forall t>0,
\end{equation}
as well. Making the observation that
$\mathbf{x}^* + t \mathbf{u}+\mathbf{o}(t) = \mathds{1}\otimes \left(x^*+t u +o(t)\right)$, (\ref{equ_02111129}) is equivalent to showing that
\begin{equation}
\label{equ_02111133}
h\left(x^*+t u +o(t)\right)=0\ \forall t>0.
\end{equation}
However, we showed previously that $u\in \textmd{Null}\left(\nabla h(x^*)'\right)$, and therefore by Proposition \ref{pro_02091833} $u\in TC(x^*,\Omega)$, as well. Therefore there exits a function $o(t)$ so that (\ref{equ_02111133}) is satisfied, which shows that indeed
$$\textmd{TC}\left(\mathbf{x}^*,\boldsymbol{\Omega}\right) = \textmd{Null}\left(\left[\nabla \mathbf{h}(\mathbf{x}^*),\mathbf{S}'\right]'\right),$$
and consequently $\textmd{TC}\left(\mathbf{x}^*,\boldsymbol{\Omega}\right)$ is a closed and convex subspace.
\end{IEEEproof}

Let $\mathbf{x}^*=\mathds{1}\otimes x^*$ denote  a  local minimizer of $(\texttt{P}_2)$. From the theory concerning optimization problems with equality constraints (see for example Chapter 3, page 15 of \cite{Varaiya72}, or Chapter 3, page 253 of \cite{Ber99}), the first order necessary conditions for $(P_2)$ ensure the existence of $\lambda_0^*\in \mathds{R}$, $\boldsymbol{\mu}^*\in \mathds{R}^N$ and $\boldsymbol{\lambda}^*\in \mathds{R}^{n\bar{N}}$ so that $\lambda_0^*\nabla \mathbf{F}(\mathbf{x}^*) + \nabla \mathbf{h}(x^*)\boldsymbol{\mu}^*+\mathbf{S}'\boldsymbol{\lambda}^*=0$. Since $\mathbf{S}$ is not full rank, and therefore the matrix $\left[\nabla \mathbf{h}(\mathbf{x}^*),\mathbf{S}'\right]$ is not full rank either, the uniqueness of $\boldsymbol{\mu}^*$ and $\boldsymbol{\lambda}^*$ cannot be guaranteed. The following result characterizes the set of Lagrange multipliers verifying the first order necessary conditions of $(\texttt{P}_2)$.

\begin{proposition}
\label{lem_18011114}
Let Assumptions \ref{assu_02091524} and \ref{assu_02091703} hold and let $\mathbf{x}^*=\mathds{1}\otimes x^*$ be a local minimizer for problem $(\texttt{P}_2)$. There exist unique vectors $\boldsymbol{\mu}^*$ and $\boldsymbol{\lambda}^*\in \textmd{Range}(\mathbf{S})$ so that $\nabla\mathbf{F}(\mathbf{x}^*)+\nabla \mathbf{h}(\mathbf{x}^*)\boldsymbol{\mu}^*+\mathbf{S}' \boldsymbol{\lambda} = 0$ for all $\boldsymbol{\lambda}\in \left\{\boldsymbol{\lambda}^*+\boldsymbol{\lambda}_{\perp}\ |\ \boldsymbol{\lambda}_{\perp}\in \textmd{Null}\left(\mathbf{S}'\right) \right\}$.
\end{proposition}
\begin{IEEEproof}
By Lemma 1\footnote{The result states that given a local minimizer $x^*$ of a function $f(x)$,  $h'\nabla f({x}^*)\geq 0$ for all $h\in TC(x^*, \Omega)$. When $ TC(x^*, \Omega)$ is a (closed, convex) subspace, orthogonality follows.} , page 50 of \cite{Varaiya72}  we have that $\nabla\mathbf{F}(\mathbf{x}^*)$ is orthogonal on $\textmd{TC}\left(\mathbf{x}^*,\boldsymbol{\Omega}\right)$ and therefore, by Proposition \ref{pro_02181835}, $\nabla\mathbf{F}(\mathbf{x}^*)$ must belong to $\textmd{Range}\left(\left[\nabla \mathbf{h}(\mathbf{x}^*),\mathbf{S}'\right]\right)$. Consequently, there exist the vectors $\boldsymbol{\mu}^*$ and $\boldsymbol{\lambda}$ so that
\begin{equation}
\label{equ_16120128}
-\nabla\mathbf{F}(\mathbf{x}^*)=\nabla \mathbf{h}(\mathbf{x}^*)\boldsymbol{\mu}^*+  \mathbf{S}' \boldsymbol{\lambda}.
\end{equation}
Noting that $\mathds{R}^{nN}$ can be written as a direct sum between the nullspace of $\mathbf{S}'$ and the range of $\mathbf{S}$, there exist the orthogonal vectors $\boldsymbol{\lambda}^*\in \textmd{Range}\left(\mathbf{S}\right)$ and $\boldsymbol{\lambda}_{\perp}\in \textmd{Null}\left(\mathbf{S}'\right)$ so that $\boldsymbol{\lambda} = \boldsymbol{\lambda}^* + \boldsymbol{\lambda}_{\perp}$.
Note that we can replace $\boldsymbol{\lambda}_{\perp}$ by any vector in $\textmd{Null}\left(\mathbf{S}'\right)$ and (\ref{equ_16120128}) still holds. The only thing left to do is to prove the uniqueness of $\boldsymbol{\mu}^*$ and $\boldsymbol{\lambda}^*$. We use a contradiction argument. Let $\tilde{\boldsymbol{\mu}}\neq \boldsymbol{\mu}^*$ and $\tilde{\boldsymbol{\lambda}}\neq   \boldsymbol{\lambda}^*$ with $\tilde{\boldsymbol{\lambda}}\in \textmd{Range}\left(\mathbf{S}\right)$ be two vectors so that (\ref{equ_16120128}) is satisfied. Hence we have that
$$-\nabla\mathbf{F}(\mathbf{x}^*)=\nabla \mathbf{h}(\mathbf{x}^*){\boldsymbol{\mu}}^*+  \mathbf{S}' \boldsymbol{\lambda}^* \textmd{ and }-\nabla\mathbf{F}(\mathbf{x}^*)=\nabla \mathbf{h}(\mathbf{x}^*)\tilde{\boldsymbol{\mu}}+\mathbf{S}' \tilde{\boldsymbol{\lambda}},$$
and therefore
$$0 = \nabla \mathbf{h}(\mathbf{x}^*) \left(\boldsymbol{\mu}^*-\tilde{\boldsymbol{\mu}}\right)+ \mathbf{S}' \left(\boldsymbol{\lambda}^*-\tilde{\boldsymbol{\lambda}}\right).$$
By Proposition \ref{pro_02091833} we have that
$$\textmd{Null}\left(\left[\nabla \mathbf{h}(\mathbf{x}^*),\mathbf{S}'\right]\right)=
\left\{\left(\mathbf{0}',\mathbf{v}'\right)' \ |\ \mathbf{v}\in \textmd{Null}\left(\mathbf{S}'\right)\right\},$$
and therefore $\boldsymbol{\mu}^*=\tilde{\boldsymbol{\mu}}$ and $\boldsymbol{\lambda}^*=\tilde{\boldsymbol{\lambda}}$ since $\boldsymbol{\lambda}^*-\tilde{\boldsymbol{\lambda}}\in \textmd{Range}\left(\mathbf{S}\right)$, and the result follows.
\end{IEEEproof}

We can now proceed with the proofs of Proposition \ref{equ_02111222} and Lemma \ref{lem_17151116}.

\begin{IEEEproof}[Proof of Proposition \ref{equ_02111222}]

By Proposition \ref{lem_18011114}, there exist two unique vector $\boldsymbol{\mu}^*$ and $\boldsymbol{\lambda}^*\in \textmd{Range}(\mathbf{L})$ so that
$$\nabla\mathbf{F}(\mathbf{x}^*)+\nabla \mathbf{h}(\mathbf{x}^*)\boldsymbol{\mu}^*+\mathbf{S}' \boldsymbol{\lambda}^* = 0.$$
Using the structure of $\nabla\mathbf{F}(\mathbf{x}^*)$, $\mathbf{h}(\mathbf{x}^*)$ and $\mathbf{S}'$, the above equation can be equivalently expressed as
\begin{equation}
\label{equ_01120920}
\nabla f_i(x^*)+\boldsymbol{\mu}^*_i\nabla h_i(x^*)+\sum_{j\in \mathcal{N}_i}\left(s_{ij}{\lambda}^*_{ij}-
s_{ji}{\lambda}^*_{ji}\right),\ i=1,\ldots,N,
\end{equation}
where $\boldsymbol{\mu}_i^*$ are  the scalar entries of $\boldsymbol{\mu}^*$ and $\boldsymbol{\lambda}^*_i = (\lambda^*_{ij})$ are the $n|\mathcal{N}_i|$-dimensional sub-vectors of $\boldsymbol{\lambda}^*$.
Summing up equations (\ref{equ_01120920}) over $i$, we obtain
$$\sum_{i=1}^N\nabla f_i(x^*)+\sum_{i=1}^N \nabla h_i(x^*) \boldsymbol{\mu}_i^*=0.$$
Equivalently,
$$\nabla f(x^*)+\nabla h(x^*) \boldsymbol{\mu}^*=0,$$
which is just the first order necessary condition for $(P_1)$. But since $\boldsymbol{\mu}^*$ must be unique, it follows that $\boldsymbol{\mu}^* = {\psi}^*$.
\end{IEEEproof}

The convergence properties of the first two the distributed algorithms depend on the spectral properties of a particular matrix; properties analyzed in the following result.
\begin{lemma}
\label{lem_17151116}
Let Assumptions \ref{assu_02091524} and \ref{assu_02091703} hold, let $\alpha$ be a positive scalar, and let $\mathbf{x}^*$ be a local minimizer of $(P_2)$. Then the eigenvalues of the matrix
{\begin{equation*}
\small
\label{equ_02111425}
\mathbf{B} = \left(
\begin{array}{ccc}
\mathbf{H} & \nabla \mathbf{h}(\mathbf{x}^*) & \mathbf{S}'\\
-\nabla \mathbf{h}(\mathbf{x}^*)' & \mathbf{0} & \mathbf{0}\\
-\mathbf{J} & \mathbf{0} & \frac{1}{\alpha}\mathbf{J}
\end{array}
\right),
\end{equation*}}
have positive real parts, where $\mathbf{H}$ is a positive definite matrix and $\mathbf{J}$ is the orthogonal projection operator on $\textmd{Null}(\bf{S}')$.
\end{lemma}
 \begin{IEEEproof}
 Let $\beta$ be an eigenvalue of $\mathbf{B}$ and let $\left(\mathbf{u}',\mathbf{v}',\mathbf{z}'\right)'\neq 0$ be the corresponding eigenvector, where $\mathbf{u}$, $\mathbf{v}$ and $\mathbf{z}$ are complex vectors of appropriate dimensions. Denoting by $\mathbf{\hat{u}}$, $\mathbf{\hat{v}}$ and $\mathbf{\hat{z}}$ the conjugates of $\mathbf{u}$, $\mathbf{v}$ and $\mathbf{z}$, respectively we have
\begin{equation}
\label{equ_16001116}
\textmd{Re}(\beta)\left(\|\mathbf{u}\|^2+\|\mathbf{v}\|^2+\|\mathbf{z}\|^2\right)=\textmd{Re}\left\{ \left(\mathbf{\hat{u}}',\mathbf{\hat{v}}',\mathbf{\hat{z}}'\right) \mathbf{B}
\left(\begin{array}{c}
\mathbf{{u}}\\
\mathbf{{v}}\\
\mathbf{z}
\end{array}
\right)
\right\}=,
\end{equation}
$$\textmd{Re}\left\{
\mathbf{\hat{u}}' \mathbf{H}\mathbf{u}+\mathbf{\hat{u}}' \mathbf{S}'\mathbf{z}-\mathbf{\hat{z}}' \mathbf{S}\mathbf{u}+\mathbf{\hat{u}}'\nabla \mathbf{h}(\mathbf{x}^*)\mathbf{v} - \mathbf{\hat{v}}'\nabla \mathbf{h}(\mathbf{x}^*)'\mathbf{{u}}+
\mathbf{\hat{z}}' \frac{1}{\alpha}\mathbf{J}\mathbf{z}
 \right\}$$
$$=\textmd{Re}\left\{
\mathbf{\hat{u}}^T \mathbf{H}\mathbf{{u}} + \mathbf{\hat{z}}' \frac{1}{\alpha}\mathbf{J}\mathbf{z}
 \right\}.$$
Since $\mathbf{J}$ is a semi-positive definite matrix and  $\mathbf{H}$ is positive definite we have that
$$\textmd{Re}(\beta)\left(\|\mathbf{u}\|^2+\|\mathbf{v}\|^2+\|\mathbf{z}\|^2\right)>0,$$
as long as $\mathbf{u}\neq 0$ or $\mathbf{z}\notin \textmd{Range}(\mathbf{S})$ and therefore $\textmd{Re}(\beta)>0$.
In the case $\mathbf{u}=0$ and $\mathbf{z}\in \textmd{Range}(\mathbf{S})$ we get
$$\mathbf{B} \left(
\begin{array}{c}
0 \\
\mathbf{v}\\
\mathbf{z}
\end{array}
\right) = \beta \left(
\begin{array}{c}
0 \\
\mathbf{v}\\
\mathbf{z}
\end{array}
\right), $$
from where we obtain
$$\nabla \mathbf{h}(\mathbf{x}^*)\mathbf{v}+\mathbf{S}'\mathbf{z}=0.$$
But from Proposition \ref{pro_02091833}, we have that $\mathbf{v}=0$ and $\mathbf{z}\in \textmd{Null}(\mathbf{S}')$ and since $\mathbf{z}\in \textmd{Range}(\mathbf{S})$ as well, it must be that $\mathbf{z}=0$. Hence we have a contradiction since we assumed that $\left(\mathbf{u}',\mathbf{v}',\mathbf{z}'\right)\neq \mathbf{0}'$ and therefore the real part of $\beta$ must be positive.
In addition, it can be easily checked that the matrix $\mathbf{B}$ has $n$ eigenvalues equal to $\frac{1}{\alpha}$ and their corresponding eigenspace is $\left\{\left({0}',{0}',\mathbf{z}'\right)'\ |\ \mathbf{z}\in \textmd{Null}\left(\mathbf{S}'\right)\right\}$.
\end{IEEEproof}

We finalize this section by providing the proofs of the theorems describing the convergence properties of Algorithms $(\texttt{A}_1)$ and $(\texttt{A}_2)$ for solving $(\texttt{P}_2)$.
\begin{IEEEproof}[Proof of Theorem \ref{thm_17401116}]
Using the  Lagrangian function defined  in (\ref{equ_02111445}), iteration (\ref{equ_18331114})-(\ref{equ_18341114})  can be equivalently expressed as
\begin{equation}
\label{equ_02111737}
\left(\begin{array}{c}
\mathbf{x}_{k+1}\\
\boldsymbol{\mu}_{k+1}\\
\boldsymbol{\lambda}_{k+1}
\end{array}
\right)=\bar{\mathbf{M}}_{\alpha}(\mathbf{x}_k,
\boldsymbol{\mu}_k,\boldsymbol{\lambda}_k),
\end{equation}
with
{\small $$ \ \bar{\mathbf{M}}_{\alpha}(\mathbf{x},
\boldsymbol{\mu},\boldsymbol{\lambda})=
\left(\begin{array}{c}
\mathbf{x}-\alpha \nabla_{\mathbf{x}}\boldsymbol{\mathcal{L}}(\mathbf{x},
\boldsymbol{\mu},\boldsymbol{\lambda})\\
\boldsymbol{\mu}+\alpha \nabla_{\boldsymbol{\mu}}\boldsymbol{\mathcal{L}}(\mathbf{x},
\boldsymbol{\mu},\boldsymbol{\lambda})\\
\boldsymbol{\lambda}+\alpha \nabla_{\boldsymbol{\lambda}}\boldsymbol{\mathcal{L}}(\mathbf{x},
\boldsymbol{\mu},\boldsymbol{\lambda})
\end{array}
\right).$$}
It can be easily checked that $\left(\mathbf{x}^*, \boldsymbol{\mu}^*,\boldsymbol{\lambda}^*+\textmd{Null}\left(\mathbf{S}'\right) \right)$ is a set of fixed points of $\bar{\mathbf{M}}_{\alpha}$. Let us now consider the transformation $\tilde{\boldsymbol{\lambda}} = \left(\mathbf{I}-\mathbf{J}\right)\boldsymbol{\lambda}$, where $\mathbf{J}$ is the orthogonal projection operator on $\textmd{Null}(\bf{S}')$. This transformation extracts the projection of $\boldsymbol{\lambda}$ on the nullspace of $\mathbf{S}'$ from $\boldsymbol{\lambda}$ and therefore $\tilde{\boldsymbol{\lambda}}$ is the error between ${\boldsymbol{\lambda}}$ and its orthogonal projection on $\textmd{Null}\left(\mathbf{S}'\right)$. Under this transformation, iteration (\ref{equ_02111737}) becomes
\begin{equation*}
\label{equ_02111739}
\left(\begin{array}{c}
\mathbf{x}_{k+1}\\
\boldsymbol{\mu}_{k+1}\\
\tilde{\boldsymbol{\lambda}}_{k+1}
\end{array}
\right)={\mathbf{M}}_{\alpha}(\mathbf{x}_k,
\boldsymbol{\mu}_k,\tilde{\boldsymbol{\lambda}}_k)
\end{equation*}
with
$${\mathbf{M}}_{\alpha}(\mathbf{x},
\boldsymbol{\mu},\tilde{\boldsymbol{\lambda}})=
\left(\begin{array}{c}
\mathbf{x}-\alpha \nabla_{\mathbf{x}}\boldsymbol{\mathcal{L}}(\mathbf{x},
\boldsymbol{\mu},\tilde{\boldsymbol{\lambda}})\\
\boldsymbol{\mu}+\alpha \nabla_{\boldsymbol{\mu}}\boldsymbol{\mathcal{L}}(\mathbf{x},
\boldsymbol{\mu},\tilde{\boldsymbol{\lambda}})\\
\left(\mathbf{I}-\mathbf{J}\right)\tilde{\boldsymbol{\lambda}}+\alpha \nabla_{\tilde{\boldsymbol{\lambda}}}\boldsymbol{\mathcal{L}}(\mathbf{x},
\boldsymbol{\mu},\tilde{\boldsymbol{\lambda}})
\end{array}
\right),$$
where we used the fact that $(\mathbf{I}-\mathbf{J})\tilde{\boldsymbol{\lambda}} = (\mathbf{I}-\mathbf{J}){\boldsymbol{\lambda}}$, $(\mathbf{I}-\mathbf{J})\mathbf{S}\mathbf{x} = \mathbf{S}\mathbf{x} $, since $\mathbf{S}\mathbf{x} \in \textmd{Range}(\mathbf{S})$, and $\bf{S}'{\boldsymbol{\lambda}} = \bf{S}'(\tilde{\boldsymbol{\lambda}}+\mathbf{J}\boldsymbol{\lambda}) =  \bf{S}'\tilde{\boldsymbol{\lambda}}$. Clearly $\left(\mathbf{x}^*,\boldsymbol{\mu}^*,\boldsymbol{\lambda}^*\right)$ is a fixed point for $\mathbf{M}_{\alpha}$ and if $\left(\mathbf{x}_k,\boldsymbol{\mu}_k,\tilde{\boldsymbol{\lambda}}_k\right)$ converges to $\left(\mathbf{x}^*,\boldsymbol{\mu}^*,\boldsymbol{\lambda}^*\right)$, we in fact show that $\left(\mathbf{x}_k,\boldsymbol{\mu}_k,\boldsymbol{\lambda}_k\right)$ converges to $\left(\mathbf{x}^*,\boldsymbol{\mu}^*,\boldsymbol{\lambda}^*+
\textmd{Null}\left(\mathbf{S}'\right)\right)$. The derivative of the mapping $\mathbf{M}_{\alpha}\left(\mathbf{x},\boldsymbol{\mu},\boldsymbol{\lambda}\right)$ at $\left(\mathbf{x}^*,\boldsymbol{\mu}^*,\boldsymbol{\lambda}^*\right)$ is given by
$$\nabla \mathbf{M}_{\alpha}\left(\mathbf{x}^*,\boldsymbol{\mu}^*,\boldsymbol{\lambda}^*\right)=
\mathbf{I}-\alpha \mathbf{B},$$
where
$$\mathbf{B} = \left(
\begin{array}{ccc}
\nabla_{\mathbf{x}\mathbf{x}}^2\boldsymbol{\mathcal{L}}
\left(\mathbf{x}^*,\boldsymbol{\mu}^*,\boldsymbol{\lambda}^*\right) & \nabla \mathbf{h}(\mathbf{x}^*) & \mathbf{L}'\\
-\nabla \mathbf{h}(\mathbf{x}^*)' & \mathbf{0} & \mathbf{0}\\
-\mathbf{L} & \mathbf{0} & \frac{1}{\alpha}\mathbf{J}
\end{array}
\right).$$
By Lemma \ref{lem_17151116} we have that the real parts of the eigenvalues of $\mathbf{B}$ are positive and therefore we can find an $\bar{\alpha}$ so that for all $\alpha \in (0.\bar{\alpha}]$  the eigenvalues of $\nabla \mathbf{M}_{\alpha}\left(\mathbf{x}^*,\boldsymbol{\mu}^*,\boldsymbol{\lambda}^*\right)$ are strictly within the unit circle. Using a similar argument as in Proposition 4.4.1, page 387, \cite{Ber99}, there exist a norm $\|\cdot\|$ and a sphere $\mathcal{S}_{\epsilon} = \left\{(\mathbf{x}',\boldsymbol{\mu}',\boldsymbol{\lambda}')'\ | \ \|(\mathbf{x}',\boldsymbol{\mu}',\boldsymbol{\lambda}')'-
\left({\mathbf{x}^*}',{\boldsymbol{\mu}^*}',{\boldsymbol{\lambda}^*}'\right)'\|<\epsilon\right\}$
for some $\epsilon >0$ so that the induced norm of $\nabla \mathbf{M}_{\alpha}\left(\mathbf{x},\boldsymbol{\mu},\boldsymbol{\lambda}\right)$ is less than one within the sphere $\mathcal{S}_{\epsilon}$. Therefore, using the mean value theorem, it follows that $\mathbf{M}_{\alpha}\left(\mathbf{x},\boldsymbol{\mu},
\boldsymbol{\lambda}\right)$ is a contraction map for any vector in the sphere $\mathcal{S}_{\epsilon}$. By invoking the contraction map theorem (see for example Chapter 7 of \cite{Istratescu1981}) it follows that
$\left(\mathbf{x}_k,\boldsymbol{\mu}_k,\tilde{\boldsymbol{\lambda}}_k\right)$ converges to $\left({\mathbf{x}^*},{\boldsymbol{\mu}^*},{\boldsymbol{\lambda}^*}\right)$ for any initial value in $S_{\epsilon}$.
\end{IEEEproof}

\begin{IEEEproof}[Proof of Theorem \ref{thm_02181546}]
First note that the assumption on the Hessian of $\boldsymbol{\mathcal{L}}
\left(\mathbf{x},\boldsymbol{\mu},\boldsymbol{\lambda}\right)$ basically means that $\mathbf{x}^*$ is a strictly local minimizer of $(P_2)$.
 Proceeding as in the case of the proof of Theorem \ref{thm_17401116}, iteration (\ref{equ_02181523})-(\ref{equ_02181525}) can be compactly expressed as
\begin{equation}
\label{equ_02182013}
\left(\begin{array}{c}
\mathbf{x}_{k+1}\\
\boldsymbol{\mu}_{k+1}\\
\boldsymbol{\lambda}_{k+1}
\end{array}
\right)=\bar{\mathbf{M}}_{\alpha,c}(\mathbf{x}_k,
\boldsymbol{\mu}_k,\boldsymbol{\lambda}_k),
\end{equation}
with
$$\ \bar{\mathbf{M}}_{\alpha,c}(\mathbf{x},
\boldsymbol{\mu},\boldsymbol{\lambda})=
\left(\begin{array}{c}
\mathbf{x}-\alpha \nabla_{\mathbf{x}}\boldsymbol{\mathcal{L}}_c(\mathbf{x},
\boldsymbol{\mu},\boldsymbol{\lambda})\\
\boldsymbol{\mu}+\alpha \nabla_{\boldsymbol{\mu}}\boldsymbol{\mathcal{L}}_c(\mathbf{x},
\boldsymbol{\mu},\boldsymbol{\lambda})\\
\boldsymbol{\lambda}+\alpha \nabla_{\boldsymbol{\lambda}}\boldsymbol{\mathcal{L}}_c(\mathbf{x},
\boldsymbol{\mu},\boldsymbol{\lambda}),
\end{array}
\right),$$
or, expressing (\ref{equ_02181525}) in terms of the error between $\boldsymbol{\lambda}_k$ and its projection on  $\textmd{Null}\left(\mathbf{S}'\right)$, we further have
\begin{equation*}
\label{equ_02181652}
\left(\begin{array}{c}
\mathbf{x}_{k+1}\\
\boldsymbol{\mu}_{k+1}\\
\tilde{\boldsymbol{\lambda}}_{k+1}
\end{array}
\right)={\mathbf{M}}_{\alpha,c}(\mathbf{x}_k,
\boldsymbol{\mu}_k,\tilde{\boldsymbol{\lambda}}_k)
\end{equation*}
with
$${\mathbf{M}}_{\alpha,c}(\mathbf{x},
\boldsymbol{\mu},\tilde{\boldsymbol{\lambda}})=
\left(\begin{array}{c}
\mathbf{x}-\alpha \nabla_{\mathbf{x}}\boldsymbol{\mathcal{L}}_c(\mathbf{x},
\boldsymbol{\mu},\tilde{\boldsymbol{\lambda}})\\
\boldsymbol{\mu}+\alpha \nabla_{\boldsymbol{\mu}}\boldsymbol{\mathcal{L}}_c(\mathbf{x},
\boldsymbol{\mu},\tilde{\boldsymbol{\lambda}})\\
\left(\mathbf{I}-\mathbf{J}\right)\tilde{\boldsymbol{\lambda}}+\alpha \nabla_{\tilde{\boldsymbol{\lambda}}}\boldsymbol{\mathcal{L}}_c(\mathbf{x},
\boldsymbol{\mu},\tilde{\boldsymbol{\lambda}})
\end{array}
\right),$$
and $\tilde{\boldsymbol{\lambda}}_k = (\mathbf{I}-\mathbf{J}){\boldsymbol{\lambda}}_k$.
 Clearly $\left(\mathbf{x}^*,\boldsymbol{\mu}^*,\boldsymbol{\lambda}^*\right)$ is a fixed point for $\mathbf{M}_{\alpha,c}$ and we have that
$$\nabla \mathbf{M}_{\alpha,c}\left(\mathbf{x}^*,\boldsymbol{\mu}^*,\boldsymbol{\lambda}^*\right)=
\mathbf{I}-\alpha \mathbf{B}_c,$$
where
$$\mathbf{B}_c = \left(
\begin{array}{ccc}
\nabla_{\mathbf{x}\mathbf{x}}^2\boldsymbol{\mathcal{L}}_c
\left(\mathbf{x}^*,\boldsymbol{\mu}^*,\boldsymbol{\lambda}^*\right) & \nabla \mathbf{h}(\mathbf{x}^*) & \mathbf{S}'\\
-\nabla \mathbf{h}(\mathbf{x}^*)' & \mathbf{0} & \mathbf{0}\\
-\mathbf{S} & \mathbf{0} & \frac{1}{\alpha}\mathbf{J}
\end{array}
\right).$$
Since $\textmd{Null}(\mathbf{S})=\textmd{Null}(\mathbf{S}'\mathbf{S})=\textmd{Null}(\mathbf{L})$, it can be easily checked that $\textmd{Null}\left(\left[\nabla \mathbf{h}(\mathbf{x}^*),\mathbf{L}'\right]'\right)=\textmd{Null}\left(\left[\nabla \mathbf{h}(\mathbf{x}^*),\mathbf{S}'\right]'\right)$. Using the assumption that $\mathbf{x}'\nabla_{\mathbf{x}\mathbf{x}}^2\boldsymbol{\mathcal{L}}
\left(\mathbf{x}^*,\boldsymbol{\mu}^*,\boldsymbol{\lambda}^*\right)\mathbf{x}>0$ for all $\mathbf{x}\in \textmd{TC}(\mathbf{x}^*,\boldsymbol{\Omega})=\textmd{Null}\left(\left[\nabla \mathbf{h}(\mathbf{x}^*),\mathbf{L}'\right]'\right)$, according to Proposition \ref{pro_03051237} there exists a positive scalar $\bar{c}$ such that $\nabla_{\mathbf{x}\mathbf{x}}^2\boldsymbol{\mathcal{L}}_c
\left(\mathbf{x}^*,\boldsymbol{\mu}^*,\boldsymbol{\lambda}^*\right)\succ 0$ for all $c\geq \bar{c}$. Therefore, by Lemma \ref{lem_17151116}, the real parts of the eigenvalues of $\mathbf{B}_c$ are positive and consequently we can find an $\bar{\alpha}(c)$ so that for all $\alpha \in (0,\bar{\alpha}(c)]$  the eigenvalues of $\nabla \mathbf{M}_{\alpha,c}\left(\mathbf{x}^*,\boldsymbol{\mu}^*,\boldsymbol{\lambda}^*\right)$ are strictly within the unit circle. Mimicking the last part of the proof of Theorem \ref{thm_17401116}, we find that $\mathbf{M}_{\alpha,c}\left(\mathbf{x},\boldsymbol{\mu},
\boldsymbol{\lambda}\right)$ is a contraction map within a sphere centered at $\left(\mathbf{x}^*,\boldsymbol{\mu}^*,\boldsymbol{\lambda}^*\right)$, and the result follows from the contraction map theorem.
\end{IEEEproof}

\subsection{Proofs of the supporting results for the distributed algorithm based on the method of multipliers}

The key result behind Algorithm $(\texttt{A}_3)$ is Proposition \ref{pro_03051249}, whose proof is provided in what follows.

\begin{IEEEproof}[Proof of Proposition \ref{pro_03051249}]
By Proposition \ref{pro_03051237} there exists a positive scalar $\bar{c}$ such that $\nabla_{\mathbf{x}\mathbf{x}}^2{\boldsymbol{\mathcal{L}}}_c\left(\mathbf{x}^*,\boldsymbol{\eta}^*\right)\succ 0$ for  all $c\geq \bar{c}$. Following the same idea as in \cite{Bertsekas1982}, for $c>0$ we define the system of equations
\begin{eqnarray}
\label{equ_03181848}
\nabla \mathbf{F}(\mathbf{x}) + \nabla \tilde{\mathbf{h}}(\mathbf{x})\tilde{\boldsymbol{\eta}} = 0,\\
\label{equ_03181849}
\tilde{\mathbf{h}}(\mathbf{x}) + \gamma \left(\mathbf{T}{\boldsymbol{\eta}} - \tilde{\boldsymbol{\eta}}\right)/c = 0.
\end{eqnarray}
Note that compared to Proposition 2.4, \cite{Bertsekas1982}, we introduced the operator $\mathbf{T}$ acting on ${\boldsymbol{\eta}}$. From the equation (\ref{equ_03181849}) we can also note that $\tilde{\boldsymbol{\lambda}}\in \textmd{Range}(\mathbf{S})$ (consider the structure of $\mathbf{T}$ and $\tilde{\mathbf{h}}(\mathbf{x})$).
By introducing the variables
$$t = \frac{\mathbf{T}\left({\boldsymbol{\eta}} - {\boldsymbol{\eta}}^*\right)}{c},\ \gamma = \frac{1}{c},$$
the system (\ref{equ_03181848})-(\ref{equ_03181849}) becomes
\begin{eqnarray}
\label{equ_03181856}
\nabla \mathbf{F}(\mathbf{x}) + \nabla \tilde{\mathbf{h}}(\mathbf{x})\tilde{\boldsymbol{\eta}} = 0,\\
\label{equ_03181857}
\tilde{\mathbf{h}}(\mathbf{x}) + t + \gamma \left({\boldsymbol{\eta}}^* - \tilde{\boldsymbol{\eta}}\right) = 0,
\end{eqnarray}
where we used the fact that $\mathbf{T}{\boldsymbol{\eta}}^* = {\boldsymbol{\eta}}^*$. For $t=0$ and $\gamma \in [0,1/\bar{c}]$, system (\ref{equ_03181856})-(\ref{equ_03181857}) has the solution $\mathbf{x} = \mathbf{x}^*$ and $\tilde{\boldsymbol{\eta}} = {\boldsymbol{\eta}}^*$. Note that thanks to the manner we defined  (\ref{equ_03181849}), $(\mathbf{x}^*, {\boldsymbol{\eta}}^*)$ is the solution of (\ref{equ_03181856})-(\ref{equ_03181857}) and not $\left(\mathbf{x}^*, {\boldsymbol{\eta}}^*+\boldsymbol{\eta}_{\perp}\right)$, with $\boldsymbol{\eta}_{\perp}=(0,\textmd{Null}(\mathbf{S}'))$. Basically the entire proof of this proposition is based on the properties of the Jacobian of (\ref{equ_03181856})-(\ref{equ_03181857}), with respect of $(\mathbf{x},\tilde{\boldsymbol{\eta}})$ at the solution $(\mathbf{x}^*, {\boldsymbol{\eta}}^*)$. This Jacobian is given by
\begin{equation}
\label{equ_03181911}
\left[\begin{array}{cc}
\nabla_{\mathbf{x}\mathbf{x}}^2\boldsymbol{\mathcal{L}}_{0}(\mathbf{x}^*,\boldsymbol{\eta}^*) & \nabla \tilde{\mathbf{h}}(\mathbf{x}^*)\\
\nabla \tilde{\mathbf{h}}(\mathbf{x}^*)' & \gamma \mathbf{I}
\end{array}
\right].
\end{equation}
Using Proposition \ref{pro_03051237},  it can be check that for any $\gamma>0$, the Jacobian defined in (\ref{equ_03181911}) is invertible. For $\gamma = 0$, however, it turns out that the nullspace of the matrix (\ref{equ_03181911})  is given by $\left\{(\mathbf{0}',\mathbf{0}',\mathbf{w}')'\ |\ \mathbf{w}\in \textmd{Null}(
\mathbf{S}')\right\}$, and therefore, unlike Proposition 2.4, \cite{Bertsekas1982},  the Jacobian is not invertible. By choosing an arbitrarily large positive scalar $c_{\textmd{max}}$ so that $\bar{c}\leq c\leq c_{\textmd{max}}$, we in fact make sure that the matrix (\ref{equ_03181911}) is invertible for all considered values of $c$.

By defining the compact set $K = \{(0,\gamma)\ |\ \gamma\in [1/c_{\textmd{max}},1/\bar{c}]\}$ and applying the implicit function theorem (with respect to a compact set), there exist $\delta>0$, $\varepsilon>0$, and unique continuously differentiable functions $\hat{\mathbf{x}}(t,\gamma)$ and $\hat{\boldsymbol{\eta}}(t,\gamma)$, defined on $S(K;\delta)$ such that
$$\left(\|\hat{\mathbf{x}}(t,\gamma)-\mathbf{x}^*\|^*+\|\hat{\boldsymbol{\eta}}(t,\gamma)-\boldsymbol{\eta}^*\|^2\right)^{1/2}<\varepsilon \ \forall (t,\gamma)\in S(K;\delta),$$
and satisfying
\begin{eqnarray}
\label{equ_03051613}
\nabla \mathbf{F}\left(\hat{\mathbf{x}}(t,\gamma)\right)+\nabla \tilde{\mathbf{h}}\left(\hat{\mathbf{x}}(t,\gamma)\right)\hat{\boldsymbol{\eta}}(t,\gamma)=0,\\
\label{equ_03051614}
\tilde{\mathbf{h}}\left(\mathbf{\hat{\mathbf{x}}(t,\gamma)}\right)
+t+\gamma \boldsymbol{\eta}^*-\gamma \hat{\boldsymbol{\eta}}(t,\gamma)=0.
\end{eqnarray}
In addition, from the continuity of the Hessian of the augmented Lagrangian and using the fact that
$$\nabla_{\mathbf{x}\mathbf{x}}^2\boldsymbol{\mathcal{L}}_{c}(\mathbf{x}^*,\boldsymbol{\eta}^*) = \nabla_{\mathbf{x}\mathbf{x}}^2\boldsymbol{\mathcal{L}}_{0}(\mathbf{x}^*,\boldsymbol{\eta}^*)+\frac{1}{\gamma} \nabla\tilde{\mathbf{h}}(\mathbf{x}^*)\nabla \tilde{\mathbf{h}}(\mathbf{x}^*)' \succ 0, \ \forall c\geq \bar{c}$$
$\delta$ and $\varepsilon$ can be chosen so that
$$\nabla_{\mathbf{x}\mathbf{x}}^2\boldsymbol{\mathcal{L}}_{0}\left(\hat{\mathbf{x}}(t,\gamma),\hat{\boldsymbol{\eta}}(t,\gamma)\right)+\frac{1}{\gamma} \nabla \tilde{\mathbf{h}}\left(\hat{\mathbf{x}}(t,\gamma)\right)\nabla \tilde{\mathbf{h}}\left(\hat{\mathbf{x}}(t,\gamma)\right)' \succ 0 \ \forall (t,\gamma)\in S(K;\delta),\ \bar{c} \leq c\leq c_{\textmd{max}}.$$

For $\bar{c} \leq c\leq c_{\textmd{max}}$ and $\|\mathbf{T}\boldsymbol{\eta}-\boldsymbol{\eta}^*\|<c\delta$ we define
$$\mathbf{x}(\boldsymbol{\eta},c) =\hat{\mathbf{x}}\left(\frac{\mathbf{T}\boldsymbol{\eta}-\boldsymbol{\eta}^*}{c},\frac{1}{c}\right),
\tilde{\boldsymbol{\eta}}(\boldsymbol{\eta},c) =\hat{\boldsymbol{\eta}}\left(\frac{\mathbf{T}\boldsymbol{\eta}-\boldsymbol{\eta}^*}{c},\frac{1}{c}\right),$$
and we obtain that for $(\boldsymbol{\eta},c)\in D$
\begin{eqnarray}
\label{equ_03061507}
\nabla \mathbf{F}\left({\mathbf{x}}(\boldsymbol{\eta},c)\right)+\nabla \tilde{\mathbf{h}}\left({\mathbf{x}}(\boldsymbol{\eta},c)\right)\tilde{\boldsymbol{\eta}}(\boldsymbol{\eta},c)=0,\\
\label{equ_03061508}
\tilde{\boldsymbol{\eta}}(\boldsymbol{\eta},c) = \mathbf{T}{\boldsymbol{\eta}}+c
\tilde{\mathbf{h}}\left({\mathbf{x}}(\boldsymbol{\eta},c)\right),\\
\label{equ_03061509}
\nabla_{\mathbf{x}\mathbf{x}}^2\boldsymbol{\mathcal{L}}_{c}\left({\mathbf{x}}(\boldsymbol{\eta},c),\boldsymbol{\eta}\right)\succ 0.
\end{eqnarray}
which basically proves the result, except (\ref{equ_03051319}) and (\ref{equ_03051321}).

The inequalities (\ref{equ_03051319})-(\ref{equ_03051321}) are based on the differentiation of (\ref{equ_03181856})-(\ref{equ_03181857}), with respect to $t$ and $\gamma$. Defining $\gamma_{\textmd{min}} = 1/c_{\textmd{max}}$, we obtain that for all $(t,\gamma)$ such $\|t\|< \delta$ and $\gamma \in [\gamma_{\textmd{min}}, 1/\bar{c}]$
\begin{eqnarray}
\label{equ_03191627}
\left[
\begin{array}{c}
\hat{\mathbf{x}}(t,\gamma)-\mathbf{x}^*\\
\hat{\boldsymbol{\eta}}(t,\gamma) - {\boldsymbol{\eta}}^*
\end{array}
\right]= \left[\begin{array}{c}
\hat{\mathbf{x}}(t,\gamma)-\hat{\mathbf{x}}(0,\gamma_{\textmd{min}})\\
\hat{\boldsymbol{\eta}}(t,\gamma) - \hat{\boldsymbol{\eta}}(0,\gamma_{\textmd{min}})
\end{array}
\right]=\\
\nonumber
\int_{0}^{1} A\left(\xi t,(\gamma -\gamma_{\textmd{min}} )\xi+\gamma_{\textmd{min}}\right)
\left[\begin{array}{cc}
0 & 0\\
-I & \hat{\boldsymbol{\eta}}(\xi t,(\gamma -\gamma_{\textmd{min}} )\xi+\gamma_{\textmd{min}})
\end{array}
\right]
\left[\begin{array}{c}
t\\
\gamma -\gamma_{\textmd{min}}
\end{array}
\right]
d\xi,
\end{eqnarray}
where
\begin{equation}
\label{equ_03191629}
A(t,\gamma) =
\left[\begin{array}{cc}
\nabla_{\mathbf{x}\mathbf{x}}^2\boldsymbol{\mathcal{L}}_{0}\left(\hat{\mathbf{x}}(t,\gamma),\hat{\boldsymbol{\eta}}(t,\gamma)\right) & \nabla \tilde{\mathbf{h}}\left(\hat{\mathbf{x}}(t,\gamma)\right)\\
\nabla \tilde{\mathbf{h}}\left(\hat{\mathbf{x}}(t,\gamma)\right)' & \gamma \mathbf{I}
\end{array}
\right]^{-1}.
\end{equation}

In the above expressions we used the fact that $\hat{\mathbf{x}}(0,\gamma)=\mathbf{x}^*$ and $\hat{\boldsymbol{\eta}}(0,\gamma) = \boldsymbol{\eta}^*$ for all $\gamma \in [\gamma_{\textmd{min}},1/\bar{c}]$ and that $A(t,\gamma)$ is well defined for all $(t,\gamma)\in D$. Note that compared to the standard case we introduced $\gamma_{\textmd{min}}$ in (\ref{equ_03191627}) to cope with the fact that $c$ is upper bounded by $c_{\textmd{max}}$.
The rest of the proof follows pretty much  identical steps as in Proposition 2.4, page 108 of \cite{Bertsekas1982}. It is based on applying the norm operator on (\ref{equ_03191627}) and on the fact that for all $(t,\gamma)\in D$, $\|A(t,\gamma)\|$ is uniformly bounded, and where we also use the fact that $\|\mathbf{T}\boldsymbol{\eta}-\boldsymbol{\eta}^*\|\leq \|\boldsymbol{\eta}-\boldsymbol{\eta}^*\|$.
\end{IEEEproof}

The relevant quantity to study convergence is $\tilde{\boldsymbol{\eta}}(\boldsymbol{\eta},c)-\boldsymbol{\eta}^*$. The following result, which is the counterpart of Proposition 2.6 of \cite{Bertsekas1982}, permits the calculation of an upper bound for this quantity that will be used for the convergence analysis.
\begin{proposition}
\label{pro_03061450}
Let Assumptions \ref{assu_02091524} and \ref{assu_02091703} hold, let $(\mathbf{x}^*,\boldsymbol{\eta}^*)$, be a local minimizer-Lagrange multipliers pair of $(P_2)$, and let $\bar{c}$ and $\delta$ as in Proposition \ref{pro_03051249}. For all $(\boldsymbol{\eta},c)$ in the set $D$ defined by (\ref{equ_03051258}), there holds
\begin{equation}
\label{equ_03061454}
\tilde{\boldsymbol{\eta}}(\boldsymbol{\eta},c)-\boldsymbol{\eta}^* = \int_{0}^1N_c\left(\boldsymbol{\eta}^*+\zeta \mathbf{T}(\boldsymbol{\eta}-\boldsymbol{\eta}^*)\right)\mathbf{T}(\boldsymbol{\eta}-\boldsymbol{\eta}^*)d\zeta,
\end{equation}
where for all $(\boldsymbol{\eta},c)\in D$, the matrix $N_c$ is given by
\begin{equation}
\label{equ_03061458}
N_c(\boldsymbol{\eta}) = \mathbf{I}-c\nabla \tilde{\mathbf{h}}\left(\mathbf{x}(\boldsymbol{\eta},c)\right)'
\left\{\nabla_{\mathbf{x}\mathbf{x}}^2{\boldsymbol{\mathcal{L}}}_c\left(\mathbf{x}(\boldsymbol{\eta},c),\boldsymbol{\eta}\right)\right\}^{-1}
\nabla\tilde{\mathbf{h}}\left(\mathbf{x}(\boldsymbol{\eta},c)\right).
\end{equation}
\end{proposition}
\begin{IEEEproof}
The proof is similar to the proof of Proposition 2.6, page 115, \cite{Bertsekas1982}, with the difference that in (\ref{equ_03061454}) we introduce the operator $\mathbf{T}$, due the definition of $\tilde{\boldsymbol{\eta}}(\boldsymbol{\eta},c)$ in (\ref{equ_03061508}). In addition, in the change of variable in (\ref{equ_03061454}), we made use of the fact that $\mathbf{T}\boldsymbol{\eta}^* = \boldsymbol{\eta}^*$. Note that for $(\boldsymbol{\eta},c)\in D$, by (\ref{equ_03061509}) we guarantee that $\nabla_{\mathbf{x}\mathbf{x}}^2{\boldsymbol{\mathcal{L}}}_c\left(\mathbf{x}(\boldsymbol{\eta},c),\boldsymbol{\eta}\right)\succ 0$, and therefore its inverse is well defined. 
\end{IEEEproof}

We are now ready to provide the proof of Theorem \ref{pro_03061521}.
\begin{IEEEproof}[Proof of Theorem \ref{pro_03061521}]
The proof follows the same lines as the proof of Proposition 2.7 of \cite{Bertsekas1982}. The proof is based on the results introduced in Proposition \ref{pro_03051249} which states that the matrix $\nabla_{\mathbf{x}\mathbf{x}}^2{\boldsymbol{\mathcal{L}}}_c\left(\mathbf{x}(\boldsymbol{\eta},c),\boldsymbol{\eta}\right)$ is invertible for all $(\boldsymbol{\eta},c)\in D$ and therefore, the matrix $N_c(\boldsymbol{\eta})$ defined in (\ref{equ_03061458}) of Proposition \ref{pro_03061450} is well defined.

Following the foot steps of the proof of Proposition 2.7 of \cite{Bertsekas1982}, we have that the eigenvalues of $N_c(\boldsymbol{\eta}^*)$ can be expressed as
$$\boldsymbol{\sigma}_i(c) = \frac{{\mathbf{e}}_i}{{\mathbf{e}}_i+c}, \ i=1,\ldots,nN,$$
and from inequality  $\bar{c}>\max\{-2{\mathbf{e}}_1,\ldots,-2{\mathbf{e}}_{nN}\}$, we have that
\begin{equation}
\label{equ_03062040}
\max_{i=1,\ldots nN}\left|\boldsymbol{\sigma}_i(c)\right|<1,\ \forall c\geq \bar{c}.
\end{equation}
By Proposition \ref{pro_03051249}, for all $(\boldsymbol{\eta},c) \in D$ we have that $\mathbf{x}(\boldsymbol{\eta},c)$ and $\tilde{\boldsymbol{\eta}}(\boldsymbol{\eta},c)$ satisfying (\ref{equ_03061507})-(\ref{equ_03061508}), are continuously differentiable, and therefore for any $\varepsilon_1>0$ there exists a $\delta_1\in (0,\delta]$ such that, for all $(\boldsymbol{\eta},c)\in D_1=\{(\boldsymbol{\eta},c)\ |\ \|\mathbf{T}\boldsymbol{\eta}-\boldsymbol{\eta}^*\|/c<\delta_1,\bar{c}\leq c\leq c_{\textmd{max}} \}$ we have
$$\|N_c(\boldsymbol{\eta})\|\leq \|N_c(\boldsymbol{\eta}^*)\|+\varepsilon_1=\max_{i=1,\ldots nN}\left|\boldsymbol{\sigma}_i(c)\right|+\varepsilon_1,$$
where the matrix $N_c(\boldsymbol{\eta})$ was defined in (\ref{equ_03061458}).

Using (\ref{equ_03061454}) of Proposition \ref{pro_03061450} together with the above inequality we have that for all $(\boldsymbol{\eta},c)\in D_1$
\begin{eqnarray}
\nonumber
\|\tilde{\boldsymbol{\eta}}({\boldsymbol{\eta}},c)-{\boldsymbol{\eta}}^*\|\leq \int_{0}^1  \left\|N_c\left(\boldsymbol{\eta}^*+\zeta \mathbf{T}(\boldsymbol{\eta}-\boldsymbol{\eta}^*)\right)\right\| \|\mathbf{T}(\boldsymbol{\eta}-\boldsymbol{\eta}^*)\|d\zeta \leq \\
\label{equ_03071729}
\leq \left(\max_{i=1,\ldots nN}\left|\boldsymbol{\sigma}_i(c)\right|+\varepsilon_1\right)\|\mathbf{T}\| \|\boldsymbol{\eta}-\boldsymbol{\eta}^*\|\leq
\left(\max_{i=1,\ldots nN}\left|\boldsymbol{\sigma}_i(c)\right|+\varepsilon_1\right)\|\boldsymbol{\eta}-\boldsymbol{\eta}^*\|,
\end{eqnarray}
where the last inequality followed from the fact that $\mathbf{T}\boldsymbol{\eta}^* = \boldsymbol{\eta}^*$ and $\|\mathbf{T}\|=1$. Keeping in mind that $\varepsilon_1$ can be chosen arbitrarily small, we have that there exists $\rho(\varepsilon_1,c) = \max_{i=1,\ldots nN}\left|\boldsymbol{\sigma}_i(c)\right|+\varepsilon_1$ so that $\rho(\varepsilon_1,c)\in (0,1)$ for all $(\boldsymbol{\eta},c)\in D_1$, and therefore
$$\|\tilde{\boldsymbol{\eta}}({\boldsymbol{\eta}},c)-{\boldsymbol{\eta}}^*\|\leq \rho(\varepsilon_1,c) \|\boldsymbol{\eta}-\boldsymbol{\eta}^*\|.$$
From the above, together with (\ref{equ_03051319}) of Proposition \ref{pro_03051249} we have that $\boldsymbol{\eta}_k\rightarrow \boldsymbol{\eta}^* $ and $\mathbf{x}(\boldsymbol{\eta}_k,c_k)\rightarrow \mathbf{x}^*$, and the results for the rates of convergence follow from (\ref{equ_03071729}).
\end{IEEEproof}

\end{document}